\documentclass[10pt, a4paper]{article}
\usepackage[utf8]{inputenc}
\usepackage{amsfonts}
\usepackage{amsmath}
\usepackage{amsthm}
\usepackage{amssymb}
\usepackage{caption}
\usepackage{comment}
\usepackage{xcolor}
\usepackage[colorlinks, backref=page]{hyperref}        
\hypersetup{citecolor=blue}
\usepackage{mathrsfs}
\usepackage[margin=1.25in]{geometry}
\usepackage[round]{natbib}
\usepackage{graphicx}
\usepackage{wrapfig}
\usepackage{cleveref}
\usepackage{authblk}

\newcommand{\A}{\mathcal A}

\newcommand{\Cc}{\mathcal C}

\newcommand{\Ee}{\mathcal E}
\newcommand{\ee}{\varepsilon}

\newcommand{\Hh}{\mathcal H}

\newcommand{\KL}{\mathrm{KL}}

\newcommand{\Ll}{\mathcal{L}}

\newcommand{\one}[1]{\mathbf{1}_{#1}}

\newcommand{\Pp}{\mathbb P}

\newcommand{\PR}{\mathcal P}

\newcommand{\R}{\mathbb R}

\newcommand{\V}{\mathcal V}

\newcommand{\X}{\mathcal X}
\newcommand{\Y}{\mathcal Y}

\DeclareMathOperator*{\aff}{aff}

\DeclareMathOperator*{\argmax}{arg\,max}
\DeclareMathOperator*{\conv}{conv}
\DeclareMathOperator*{\inter}{int}
\DeclareMathOperator*{\lin}{lin}
\DeclareMathOperator*{\relint}{rel\,int}
\DeclareMathOperator*{\Span}{Span}
\DeclareMathOperator*{\supp}{Supp}

\newtheorem{definition}{Definition}

\newtheorem{lemma}{Lemma}
\newtheorem*{lemma*}{Lemma}
\newtheorem{proposition}{Proposition}
\newtheorem{theorem}{Theorem}
\newtheorem{corollary}{Corollary}

\newtheorem{manin}{}
\newenvironment{manual}[1]{
\renewcommand\themanin{#1}
  \manin
}{\endmanin}

\title{Optimal e-value testing for properly constrained hypotheses\\\vspace{.2cm}}
\author{Eugenio Clerico\footnote{Correspondence at: eugenio.clerico@gmail.com}}
\affil{\small{\textit{Universitat Pompeu Fabra, Barcelona, Catalunya (Spain)}}}
\date{\vspace{-1cm}}
\begin{document}
\maketitle

\begin{abstract}\noindent
Hypothesis testing via e-variables can be framed as a sequential betting game, where a player each round picks an e-variable. A good player's strategy results in an effective statistical test that rejects the null hypothesis as soon as sufficient evidence arises. Building on recent advances, we address the question of restricting the pool of e-variables to simplify strategy design without compromising effectiveness. We extend the results of \cite{clerico2024optimality}, by characterising optimal sets of e-variables for a broad class of non-parametric hypothesis tests, defined by finitely many regular constraints. As an application, we discuss this notion of optimality in algorithmic mean estimation, including for heavy-tailed random variables.
\end{abstract}

\section{Introduction}\label{sec:introduction}
Hypothesis testing is the branch of statistics concerned with verifying if an observed data set is consistent with a given theory, often referred to as the \emph{null hypothesis}. Traditionally, this is formulated in terms of \emph{p-values}, which represent the probability of obtaining an empirical test statistic at least as ``extreme'' as the one observed, assuming that the null hypothesis holds. However, most classical p-value methods are designed for single, fixed-sample experiments. In modern research, where data are often collected iteratively or experiments are repeated, recalculating p-values with additional data can lead to false positives and misleading conclusions. To overcome these limitations, recent developments propose using \emph{e-values} as a more reliable alternative for sequential and adaptive hypothesis testing. 

E-values quantify evidence against a hypothesis and allow flexible data collection strategies, such as optional stopping or continuation, and post hoc significance level tuning. While related ideas were developed decades ago by early works \citep{wald1945sequential, darling1967confidence, lai1976confidence, siegmund1978estimation}, interest in e-values surged with a series of papers whose first preprints appeared over the space of few months in 2019 \citep{wasserman2020universal, shafer2021testing, vovk2021evalues, grunwald2024safe}. In the present work, we focus on a particular case of testing with e-values, which can be phrased in terms of a sequential \emph{betting game} \citep{ramdas2023game}, adopting a game-theoretic perspective on probability recently developed by \cite{shafer2001probability,shafer2019game}.  In this game, at each round the player selects a \emph{single round e-variable}, namely a non-negative function with an expected value of at most one under the null hypothesis. The player then receives a reward, determined by the value of the picked e-variable at the next observed data point. If a significant grow of the cumulative reward is observed, the null hypothesis can be confidently rejected, as standard martingale analysis shows that this is unlikely to happen under the null hypothesis. We refer to \cite{ramdas2024hypothesis} for an overview on e-value testing and game-theoretic interpretations.

A testing approach is effective if it leads to the rejection of a hypothesis as soon as there is enough evidence against it in the data. In our framework, this can be achieved if the player adopts a sound strategy, which yields high rewards whenever possible. The theoretical foundation for selecting such strategies in sequential games lies in \emph{online learning} theory \citep{cesabianchi2006prediction, orabona2023modern}. To simplify strategy design and leverage online learning theoretical tools, it is often convenient to restrict the pool of e-variables among which the player can pick each round. In order to make this point more explicit, \citet{clerico2024optimality} introduced the concept of \emph{restricted} testing game. A key question is whether such restriction is detrimental for the test, by preventing the selection of strategies that could yield higher rewards. Tor many instantiations of ths \textit{testing by betting} framework the answer is \textit{no}, provided the restriction is well chosen. Moreover, in some cases, there exists an \emph{optimal} class of e-variables to which one should restrict the game, informally the smallest set within which any good strategy must choose the e-variables. These ideas were developed in \citet{clerico2024optimality}, in the context of a simple e-variable testing problem that was motivated by mean estimation for a random variable bounded in $[0,1]$. This work extends these results to a wider class of hypothesis testing problems, by characterising the optimal set of single round e-variables for non-parametric hypotheses  defined by a finite number of suitably regular constraints. This class is broad enough to capture meaningful problems, while allowing us to use standard analysis techniques, without relying on advanced tools from functional analysis and measure theory.

Notions of optimality for single e-variables, such as \emph{log-optimality} \citep{koolen2022log, larsson2024numeraire, grunwald2024safe} and \emph{admissibility} \citep{ramdas2022admissible}, have been previously extensively explored in the sequential hypothesis testing literature. However, as highlighted by \citet{clerico2024optimality}, the key distinction in our approach is the emphasis on optimality at the level of sets of e-variables, rather than focusing on individual e-variables. More explicitly, we show that for a class of constrained hypotheses, the whole set of single-round e-variables is dominated by the set of all maximal e-variables, namely e-variables that cannot be dominated by any other e-variables. We remark here that the notion of \textit{maximality} of an e-variable is essentially equivalent to that of \textit{admissibility} in classical statistics. Somehow our result can be seen as an existence and characterisation theorem of a \textit{minimal complete} class of admissible e-variables. However, in the e-variable literature, it is common to define admissibility with respect to an \textit{alternative} hypothesis (see \citealp{ramdas2022admissible}), while the current paper does not rely on any given alternative. For this reason we prefer to keep the term \textit{maximal} (typical in the literature on posets), as done in \cite{clerico2024optimality}.

As a final remark, we stress once more that the results of this paper only focus on optimality of testing by betting via e-variables. This is by far not the most general type of sequential testing.  An important open question is to characterise optimality for more general forms of sequential testing. More concretely, this paper is only one first step towards characterising all the admissible sequential tests, as it characterises the form of maximal single-round e-variables.

\subsection{Structure}
This paper is organised as follows. In \Cref{sec:AHT}, we briefly review e-value hypothesis testing by betting. \Cref{sec:majopt} introduces the key definitions and basic properties of majorising and optimal classes of e-variables, while \Cref{sec:PCH} defines the concept of properly constrained hypothesis, which serves as the central framework for this work. \Cref{sec:dual} discusses the dual class of e-variables for properly constrained hypotheses and states our main result: the optimality of the dual class (\Cref{thm:dualopt}), which will be proved throughout the subsequent sections. Specifically, we begin by addressing the case of a finite sample space in \Cref{sec:finite}, and then introduce the notion of matching sets in \Cref{sec:matching}, which will be an essential tool to extend the optimality results to general closed sample spaces in \Cref{sec:optimal}. In \Cref{sec:extensions}, we broaden our results, by addressing tests of finitely (but not properly) constrained hypotheses, and by considering hypotheses defined via inequality constraints. In \Cref{sec:AME}, our results are applied to the problem of deriving confidence sequences for the mean of (bounded and heavy-tailed) random variables. Lastly, \Cref{sec:conclusion} provides concluding remarks and outlines potential directions for future research.

\subsection{Notation}
Before delving into the main discussion, we first introduce the key notations and conventions used in this paper. We will often make use of topological notions (e.g., continuity, closedness, compactness, Borel measurability) on $\R^n$. The underlying topology is always implicitly assumed to be the standard Euclidean one. Throughout the whole paper, $\X$ denotes a closed set in $\R^n$. 

For an arbitrary set $\Y\subseteq\R^n$, we let $\PR_\Y$ be the set of Borel probability measures on $\R^n$ whose support is included in $\Y$. If $\X\subseteq\R^n$ is closed, with a slight abuse of notation we can identify $\PR_\X$ with the set of all Borel probability measures on $\X$ (with respect to the induced topology).
For $P\in\PR_\Y$ and $\phi:\Y\to[0,+\infty)$ a Borel function, $\langle P, \phi(Y)\rangle$ (or more compactly $\langle P, \phi\rangle$) denotes the expectation of $\phi$ under $Y\sim P$. For $y\in\Y$, $\delta_y$ is the Dirac mass on $y$, namely $\langle\delta_y,\phi\rangle = \phi(y)$. For $\Phi:\Y\to\R^m$, let $\|\Phi\|_1:\Y\to\R$ denote the mapping $y\mapsto\|\Phi(y)\|_1$, with $\|\cdot\|_1$ the 1-norm in $\R^m$. If $\Phi$ is Borel and $\langle P, \|\Phi\|_1\rangle $ is finite, then the expectation of $\Phi$ under $P$ is well defined (in $\R^m$) and we still denote it as $\langle P, \Phi\rangle$. As we will often define classes of probability measures in terms of the expectation with respect to some fixed Borel function $\Phi:\Y\to\R^m$, it is convenient to introduce the notation $\PR_\Phi$ for the set of all measures in $\PR_\Y$ for which $\langle P, \|\Phi\|_1\rangle$ is finite. 

For $S\subseteq\X$, and $\Phi:\X\to\R^m$, we let $\Phi\big|_S:S\to\R^m$ be the restriction of $\Phi$ to $S$. We denote as $v\cdot u$ the scalar product when $u$ and $v$ are Euclidean vectors. Hence, for $\Phi:\X\to\R^m$ and $\lambda\in\R^m$, $\lambda\cdot\Phi$ is a real function $\X\to\R$, mapping $x$ to $\lambda\cdot\Phi(x)$. We denote as $\Span\Phi$ the linear span of the set $\{\Phi_1,\dots,\Phi_m\}$ in the vector space of real functions on $\X$ (where $\Phi_i$ is the $i$-th component of $\Phi$), namely $\Span\Phi = \{\lambda\cdot\Phi\,,\,\lambda\in\R^m\}$. For $\Phi:\X\to\R^m$ and $S\subseteq\X$, $\Phi(S)$ is the image of $S$ under $\Phi$, namely $\{\Phi(x)\,:\,x\in S\}$. In particular, $\Phi(\X)$ refers to the image of $\Phi$.

Given a set $\Y\subseteq\R^n$, we denote as $\conv\Y$ its convex hull, and as $\aff\Y$ its affine hull. We use the notation $\lin\Y$ for the linear hull (span) of $\Y$. We denote as $\inter\Y$ the interior of $\Y$ and as $\relint\Y$ its relative interior. See \Cref{app:convex} for definitions from elementary convex analysis.

Sequences are denoted as $(x_t)_{t\geq T_0}$, where $t$ is an integer-valued index and $T_0$ its smallest value. For high-probability statements, $\Pp$ expresses probability with respect to all the involved randomness. For instance, if $(X_t)_{t\geq 1}$ is a sequence of independent draws from $P\in\PR_\X$, we write $\Pp\big(X_t\geq 0\,,\,\forall t\geq 1\big)$.

\section{Algorithmic hypothesis testing}\label{sec:AHT}

For the whole paper, $\X$ denotes a closed set in $\R^n$, endowed with the Borel sigma-field. We let $\PR_\X$ be the set of Borel probability measures on $\X$. 

\begin{definition}[Hypotheses and e-variables]
A \emph{hypothesis} (on $\X$) is a non-empty subset $\Hh$ of $\PR_\X$. Given a hypothesis $\Hh$, an \emph{e-variable} (with respect to $\Hh$) is a non-negative Borel function $E:\X\to[0,+\infty)$, such that $\langle P, E\rangle\leq 1$ for any $P\in\Hh$. We denote as $\Ee_\Hh$ the set of all the e-variables with respect to $\Hh$. We call \emph{e-class} any subset of $\Ee_\Hh$.
\end{definition}

We remark that we are here focused only on \textit{single-round} hypotheses and e-variables. In particular, we are not considering hypotheses on a whole sequence of observations (in $\X^{\infty}$, as done for instance in \citealp{ramdas2022admissible}), but only on $\X$. In a way, we will be given for granted that the data come from as independent draws from some distribution on $\X$. 

Over the last few years, e-variables have become a fundamental tool in hypothesis testing, as they provide a practical and flexible alternative to the traditional p-value-based methods, allowing to overcome several
known limitations of conventional testing procedures \citep{wang2022false,grunwald2024safe, grunwald2024beyond}. E-variable testing is usually framed in terms of a sequential betting game, where at each round a player picks an e-variable \citep{ramdas2024hypothesis}. Using the terminology of \cite{clerico2024optimality}, we define the following \emph{testing game}.

\begin{definition}[Testing game]\label{def:game}
    Fix a hypothesis $\Hh$ on $\X$, and an e-class $\Ee\subseteq\Ee_\Hh$. We call \emph{$\Ee$-restricted testing game} (on $\Hh$) the following sequential procedure. Each round $t\geq 1$, a player 
    \begin{itemize}\setlength{\itemsep}{0pt}
        \item measurably\footnote{We assume that the choice of $E_t$ is measurable, namely the mapping $(x_1,\dots x_t)\mapsto E_t(x_t)$ is Borel measurable.} selects $E_t\in\Ee$, based solely on the previous observations $x_1,\dots, x_{t-1}$;
        \item observes a new data point $x_t\in\X$;
        \item earns a reward $\log E_t(x_t)\in[-\infty, +\infty)$. 
    \end{itemize}
    If $\Ee=\Ee_\Hh$, we speak of \emph{unrestricted} testing game.
\end{definition}
If the data points observed during the game are independent draws from  $P \in \Hh$, the player's cumulative reward is unlikely to grow excessively. This is formalised by the next statement, whose proof (see, e.g., \citealp{ramdas2023game, ramdas2024hypothesis, clerico2024optimality}) follows directly from Ville's inequality \citep{ville1939etude}, a uniform upper bound for non-negative super-martingales.

\begin{proposition}\label{prop:game}
    Let $\Hh$ be a hypothesis on $\X$ and consider a sequence $(X_t)_{t\geq 1}\subseteq\X$ of independent draws from some $P\in\Hh$. Fix $\delta\in(0,1)$ and $\Ee \subseteq \Ee_\Hh$. Consider an $\Ee$-restricted testing game on $\X$, where the player observes the sequence $(X_t)_{t\geq 1}$. Denote the player's cumulative reward at round $T$ as $R_T=\sum_{t=1}^T\log E_t(X_t)$. Then, we have
    $$\Pp\big(R_T\leq\log\tfrac{1}{\delta}\,,\;\forall T\geq 1\big)\geq 1-\delta\,.$$
\end{proposition}

The cumulative reward of the player can be interpreted as evidence accumulated against $\Hh$. Specifically, \Cref{prop:game} justifies a sequential testing procedure, where the null hypothesis (asserting that the data are generated from a $P \in \Hh$) is rejected whenever the player's cumulative reward exceeds the threshold value of $\log(1/\delta)$. Here, $\delta\in(0,1)$ denotes the type I confidence level, meaning that if the null hypothesis is true, it will be rejected with probability at most $\delta$. It is worth noticing that the high-probability inequality in \Cref{prop:game} holds uniformly for any data set size $T$. This ensures that when data are collected sequentially, optional stopping or continuation are allowed, an advantage not afforded by standard p-value methods. For a comprehensive overview of sequential testing by betting, we refer to Chapter 6 of \cite{ramdas2024hypothesis}.

It is clear that a player always selecting the constant function $1$ will never reject a hypothesis, as its cumulative reward is always null. To get an effective test, we want the cumulative reward to be ``as large as possible'', so that the null hypothesis is rejected whenever there is sufficient evidence against it in the observed data. It is crucial to carefully determine the pool $\Ee$ of e-variables from which the player can pick: a poor choice may exclude useful e-variables, while there is no advantage in keeping redundant ones (e.g., the constant function $1/2$). A well chosen $\Ee$ can simplify the problem and enhance the design of effective betting strategy, without compromising the strength of the test. This point was addressed by \cite{clerico2024optimality}, with the introduction of the notions of majorising and optimal e-classes, which are the subject of the next section.

\section{Majorising e-classes and optimal e-class}\label{sec:majopt}

We now present the definitions and basic properties of majorising and optimal e-classes, as introduced in \cite{clerico2024optimality}. 
In the following, fix a closed $\X\subseteq\R^n$ and a hypothesis $\Hh$ on $\X$. 
We consider the standard partial ordering for the real functions on $\X$: given two functions $f$ and $f'$ from $\X$ to $\R$, we say that $f$ \emph{majorises} $f'$, and write $f\succeq f'$, if $f(x)\geq f'(x)$ for all $x\in\X$. If $f\succeq f'$ and $f\neq f'$, $f$ is a \emph{strict majoriser} of $f'$, and we write $f\succ f'$. We will also use the symbols $\preceq$ and $\prec$, with obvious meaning.

\begin{definition}[Majorising and optimal e-classes]\label{def:maj}
    An e-class $\Ee\subseteq\Ee_\Hh$ is a \emph{majorising} e-class when, for any $E\in\Ee_\Hh$, we can find an e-variable $E'\in\Ee$ such that $E'\succeq E$. If a majorising e-class is contained in every other majorising e-class, it is called \emph{optimal}.
\end{definition}
It is clear that $\Ee_\Hh$ is itself a majorising e-class. 
Conversely, an optimal e-class may not exist. If it does exist, it is unique and corresponds to the intersection of all majorising e-classes. Next, we provide a necessary and sufficient condition for the existence of the optimal e-class, based on the notion of maximality for e-variables. For the proof, see Lemmas 3 and 4 in \cite{clerico2024optimality}.
\begin{definition}[Maximal e-variables]A \emph{maximal} e-variable is an e-variable $E\in\Ee_\Hh$ that has no strict majoriser, namely there is no $E'\in\Ee_\Hh$ such that $E'\succ E$.
\end{definition}
\begin{lemma}\label{lemma:maxopt}
    Every majorising e-class includes all maximal e-variables. An optimal e-class exists if, and only if, there is a majorising e-class whose elements are all maximal. In particular, if an optimal e-class exists, it is unique and it corresponds to the set of all maximal e-variables.
\end{lemma}

\begin{corollary}\label{cor:nonex}
     If, for some $x_0\in\X$, $\sup_{P\in\Hh}P(\{x_0\})=0$, there is no optimal e-class for $\Hh$.
\end{corollary}
\begin{proof}
    Fix $E\in\Ee_\Hh$ and let $\one{x_0}:\X\to\R$ be $1$ on $x_0$ and $0$ everywhere else. Then, for any $P\in\Hh$, we have $\langle P, E+\one{x_0}\rangle = \langle P, E\rangle \leq 1$, and so $E+\one{x_0}\in\Ee_\Hh$. Since $E+\one{x_0}\succ E$, $E$ is not maximal. We deduce that the set of all maximal e-variables is empty, and we conclude by \Cref{lemma:maxopt}.
\end{proof}

Majorising and optimal e-classes play a fundamental role in our analysis of sequential hypothesis testing. Restricting the game of \Cref{def:game} to a majorising e-class does not hinder the player's performance. To illustrate this, consider two players, Alice and Bob, observing the same sequence of data points. Bob plays an unrestricted testing game, while Alice's choices are restricted to a majorising e-class $\Ee$. Assume that Bob makes his move first, selecting an e-variable $E_t$. Alice, aware of Bob's choice, can then pick $E_t' \in \Ee$ such that, for any possible observation $x_t$, $E_t'(x_t) \geq E_t(x_t)$. Thus, Alice can opt for a strategy ensuring that her cumulative reward is always at least as large as Bob's, even though her choices are restricted to $\Ee$. Now, suppose that $\Ee$ is not just majorising, but optimal. In this case, if Bob chooses an e-variable $E_t$ outside of $\Ee$, Alice can select an alternative $E_t' \in \Ee$ such that $E_t' \succ E_t$. Such $E_t'$ provides rewards at least as good as those of $E_t$ for all possible values of the next observation, and strictly better for at least one value of $x_t$. On the other hand, if Alice picks first, the only way for Bob to be absolutely sure that his reward will not be worse than Alice's is to select the same e-variable, and thus choose from the optimal e-class. This is because Alice's choice is maximal (\Cref{lemma:maxopt}), meaning no other e-variable can majorise it. In this sense, the optimal e-class simplifies designing effective betting strategies, by eliminating all redundant and possibly detrimental e-variables.

We conclude this section with some positive result on the existence of the optimal e-class.

\begin{lemma}\label{lemma:finopt}
    If $\X$ is finite and $\sup_{P\in\Hh}P(\{x\})>0$ for every $x\in\X$, the optimal e-class exists.
\end{lemma}
\begin{proof}
    Let $d<+\infty$ be the cardinality of $\X$, and write $\X = \{x_1,\dots,x_d\}$. Any $E\in\Ee_\Hh$ and $P\in\Hh$ can be seen as elements of $\R^d$, with $\langle P, E\rangle$ being the standard dot product. For the rest of this proof we will always consider the standard Euclidean topology on $\R^d$ to define continuity, closedness and compactness. First, note that $\Ee_\Hh = \bigcap_{P\in\Hh}\{E\succeq 0\,:\,\langle P, E\rangle \leq 1\}$ is closed, as the intersection of closed sets. Moreover, it is also bounded, and thus compact. Indeed, for any $x\in\X$ and $E\in\Ee_\Hh$, we have $E(x)\sup_{P\in\Hh}P(\{x\})\leq\sup_{P\in\Hh}\langle P, E\rangle \leq 1$ and $\sup_{P\in\Hh}P(\{x\})>0$.

    By \Cref{lemma:maxopt}, it is enough to show that the set of all maximal e-variables is a majorising e-class, or equivalently that every e-variable is majorised by a maximal e-variable. So, fix $E_0\in\Ee_\Hh$. The set $\Ee_0 = \{E\in\Ee_\Hh\,:\,E\succeq E_0\}\subseteq\Ee_\Hh$ is closed, hence compact. Moreover, the mapping $E\mapsto E(x_1)$ is continuous on $\Ee_\Hh$. We can thus find $E_1 \in \argmax_{E\in\Ee_0}E(x_1)$. We iterate this procedure up to $i=d$, with $\Ee_i = \{E\in\Ee_\Hh\,:\,E\succeq E_i\}$ and $E_i \in \argmax_{E\in\Ee_{i-1}}E(x_i)$. By construction, $E_d\succeq E_0$. Consider any e-variable $E^\star\succeq E_d$ and fix any integer index $i\in [1,d]$. Then, $E^\star\succeq E_d\succeq E_{i-1}$, which by the definition of $E_i$ implies that $E^\star(x_i)\leq E_i(x_i)\leq E_d(x_i)$, and so $E^\star(x_i)=E_d(x_i)$. Thus, $E^\star=E_d$ and $E_d$ is maximal.
\end{proof}

Although we will not make use of it, we state next an extension of the lemma above, for countable $\X$. A proof, via the construction of a transfinite sequence, is detailed in \Cref{app:proofs}.

\begin{lemma}\label{lemma:countopt}
    If $\X$ is countable and $\sup_{P\in\Hh}P(\{x\})>0$ for all $x\in\X$, the optimal e-class exists.
\end{lemma}

One might think that if $\sup_{P\in\Hh}P(\{x\})>0$ for every $x\in\X$ the optimal e-class always exists. However, this is not always the case, as shown by the next result (whose proof is deferred to \Cref{app:proofs}). Remarkably, this shows that allowing the e-variables to take the value $+\infty$ would not be enough to guarantee the existence of the optimal e-class.

\begin{proposition}\label{prop:nonex}
    Let $\X=[0,1]$. Consider the hypothesis
    $$\Hh = \big\{P\in\PR_{[0,1]}\,:\,P(\{0\})\geq 1/2\big\}\cup\{U_{[0,1]}\}\,,$$
    with $U_{[0,1]}$ the uniform distribution on $[0,1]$. $\Hh$ does not admit an optimal e-class.
\end{proposition}

\section{Properly constrained hypotheses}\label{sec:PCH}

In the remainder of this paper, we will mainly focus on a specific kind of hypotheses, for which we aim to fully characterise the optimal e-class. Notably, \cite{clerico2024optimality} examined the simple case where $\X = [0,1]$ and $\Hh$ is the set of probability measures on $\X$ with a fixed mean. More explicitly, for $\mu\in(0,1)$, define $$\Hh_\mu = \{P\in\PR_{[0,1]}\,:\,\langle P,X\rangle = \mu\}$$. Testing for $\Hh_\mu$ means that we are testing if the mean of a bounded random variable is $\mu$. \cite{clerico2024optimality} showed that the optimal e-class exists and is presisely
$$\Ee_\mu = \big\{x\mapsto 1+\lambda(x-\mu)\,:\,\lambda\in[1/(\mu-1), 1/\mu]\big\}$$ (see \Cref{sec:meanbounded} for more details). 

We remark that the hypothesis $\Hh_\mu$ is an example of a \textit{constrained} hypothesis, namely a hypothesis that can be defined via the linear constraint $\langle P, X-\mu\rangle =0$. Here, we discuss the existence and characterisation of the optimal e-class for hypothesis of a similar form. We consider the general setting where $\X$ is a generic closed set in $\R^n$ and the hypothesis $\Hh$ is defined in terms of the expectation of a continuous function $\Phi: \X \to \R^m$. As an application (see \Cref{sec:heavy}) we will for instance be able to characterise the optimal e-class for testing the mean of heavy tail random variables. Since the approach in \cite{clerico2024optimality} is specifically designed for their simpler setting, we will have to develop more sophisticated proof techniques to handle the broader framework that we consider in this work.

In the following, fix a closed set $\X\subseteq\R^n$. We recall that, for $\Phi:\X\to\R^m$, $\PR_\Phi$ is the set of measures $P$ in $\PR_\X$ such that $\langle P, \|\Phi\|_1\rangle<+\infty$.

\begin{definition}[Constraints]
    Let $\Hh$ be a hypothesis on $\X$. If a continuous $\Phi:\X\to\R^m$ satisfies $$\Hh = \{P\in\PR_\Phi\;:\;\langle P, \Phi\rangle=0\}\,,$$ we call $\Phi$ a \emph{constraint} for $\Hh$. Let $\Cc$ denote the convex hull of $\Phi(\X)$. If $0\in\relint\Cc$ we say that $\Phi$ is a \emph{proper} constraint. If moreover $0\in\inter\Cc$, then we call $\Phi$ a \emph{minimal} constraint.
\end{definition}

\begin{definition}[Finitely and properly constrained hypotheses]
    A hypothesis is called \emph{finitely constrained} if it admits a constraint. It is \emph{properly constrained} if it admits a proper constraint.
\end{definition}
Next, we comment on a few properties of finitely and properly constrained hypotheses. All proofs (often short and building on elementary results from real algebra and convex analysis) are omitted here, and deferred to \Cref{app:proofs}.

Any finitely constrained $\Hh$ is convex, as $\PR_\Phi$ is convex and $\Hh$ is a level set of the convex functional $P\mapsto\langle P, \Phi\rangle$ on $\PR_\Phi$. 

The reason for focusing on properly constrained hypotheses is that if $0$ is on the (relative) boundary of $\Cc$, then one can show that all $P\in\Hh$ need to be supported on the relative boundary (see \Cref{lemma:convP}). In particular, \Cref{cor:nonex} states that in such case the optimal e-class does not exist. However, in such cases one can restrict $\X$ to a smaller set, and obtain a proper constraint from a non-proper one (this will be later formalised in \Cref{lemma:X0}). On the other hand, if a hypothesis is properly constrained, it is reach enough to put mass on every $x$. This is ensured by the next lemma (proved in \Cref{app:proofs}). 
\begin{lemma}\label{lemma:finicovpro}
    Let $\Hh$ be a finitely constrained hypothesis on $\X$. $\Hh$ is properly constrained if, and only if, $\sup_{P\in\Hh}P(\{x\})>0$ for all $x\in\X$. Moreover, if $\Hh$ is properly constrained, every constraint $\Phi$ for $\Hh$ is a proper constraint and there is $A:\X\to[1, +\infty)$ such that $E\preceq A$ for any $E\in\Ee_\Hh$.
\end{lemma}
Not only we can find a measure in $\Hh$ that put mass on every $x$ if $\Hh$ is properly constrained, we can also ask for such measure to have finite support, a property that will turn out to be useful for our analysis.
\begin{lemma}\label{lemma:prored}
    Let $\Hh$ be a properly constrained hypothesis on $\X$. Then, for every $x\in\X$ there exists a $P\in\Hh$ whose support has finite cardinality and such that $x\in\supp P$.
\end{lemma}

In general, proper constraints might be redundant, in the sense that their components might be linear dependent. When this is not the case, they are minimal constraints. This linear independence of the components of minimal constraints will be a key tool for our results, as it will often allow to reduce the problem to analysis of finite-dimensional spaces.
\begin{lemma}\label{lemma:prominiff}
    Let $\Hh$ be a properly constrained hypothesis and $\Phi$ a proper constraint. $\Phi$ is minimal if, and only if, all its components are linearly independent scalar functions.
\end{lemma}
The above results suggest that, if we have a proper constraint, it can be transformed in a minimal constraint by eliminating some of its components. In particular, we expect that we can construct a minimal constraint for every properly constrained hypothesis. This is indeed the case.
\begin{lemma}\label{lemma:prominc}
    Every properly constrained hypothesis admits a minimal constraint.
\end{lemma}

Last, we conclude with a characterisation of properly constrained hypotheses.
\begin{lemma}\label{lemma:profull}
    Let $\Hh$ be a finitely constrained hypothesis on $\X$. $\Hh$ is properly constrained if, and only if, there is $P\in\Hh$ such that $\supp P=\X$. 
\end{lemma}

\section{Dual e-class}\label{sec:dual}
For any properly constrained hypothesis $\Hh$ on a closed $\X\subseteq\R^n$, we can define a \emph{dual} e-class, which will play a fundamental role in our analysis. Indeed, it will turn out that it coincides precisely with the optimal e-class. First, we introduce the following notation, which will be used throughout the rest of the paper. For a function $\Phi:\X\to\R^m$, we let
$$\Lambda_{\Phi} = \left\{\lambda\in\R^m\;:\;\sup_{x\in \X}\lambda\cdot\Phi(x)\leq 1\right\}\,.$$

\begin{definition}\label{def:dual}
    Let $\Hh$ be a properly constrained hypothesis and $\Phi$ a constraint. The \emph{dual e-class} of $\Hh$ is 
    $$\Ee_\Hh^\vee = \left\{1-\lambda\cdot\Phi\;:\;\lambda\in\Lambda_\Phi\right\}\,.$$
\end{definition}

It is easy to check that $\Ee_\Hh^\vee$ is an e-class. Indeed, any $E\in\Ee_\Hh^\vee$ is non-negative (by the definition of $\Lambda_\Phi$), continuous, and $\langle P, E\rangle =1$ for all $P\in\Hh$. Moreover, $\Ee_\Hh^\vee$ is well defined, in the sense that it is independent of the specific constraint $\Phi$ used in its definition. To see this, recall that the span of $\Phi:\X\to\R^m$ is $\Span\Phi=\{\lambda\cdot\Phi\,:\,\lambda\in\R^m\}$, and note that $\Ee_\Hh^\vee=\{1-\phi\succeq 0\,:\,\phi\in\Span\Phi\}$. The next result implies that this set is the same for any constraint $\Phi$ for $\Hh$.

\begin{proposition}\label{prop:genspan}
    Let $\Hh$ be a properly constrained hypothesis on $\X$, with $\Phi:\X\to\R^m$ and $\Phi':\X\to\R^{m'}$ two constraints. Then, $\Span\Phi = \Span\Phi'$. 
\end{proposition}

When $\X$ has finite cardinality, \Cref{prop:genspan} is a direct consequence of \Cref{lemma:finspan} (stated in the next section). For a general closed $\X\subseteq\R^n$, we defer the proof to the end of \Cref{sec:matching}.

We now state the main result of this work: for any properly constrained hypothesis, the optimal e-class exists and coincides with the dual e-class.

\begin{theorem}\label{thm:dualopt}
    Let $\Hh$ be a properly constrained hypothesis. The dual e-class is the optimal e-class.
\end{theorem}

We will first prove the optimality of the dual e-class for the case where $\X$ has finite cardinality (\Cref{prop:findualopt}). The notion of matching set, which we will introduce in \Cref{sec:matching}, will enable us to build on this preliminary result to establish the optimality of $\Ee_\Hh^\vee$ for compact $\X$ (\Cref{prop:compdualopt}). Finally, we will prove the general case of a closed $\X$ at the end of \Cref{sec:optimal}.

For now, we start by checking that all the elements of $\Ee_\Hh^\vee$ are maximal.

\begin{lemma}\label{lemma:dualmax}
    Let $\Hh$ be a properly constrained hypothesis. Every $E\in\Ee_\Hh^\vee$ is maximal.
\end{lemma}
\begin{proof}
    Fix $E\in\Ee_\Hh^\vee$, and consider an e-variable $\hat E\succeq E$. Fix any $x\in\X$. By \Cref{lemma:finicovpro}, there is a $P\in\Hh$ with an atom on $x$, and the fact that $E\in\Ee_\Hh^\vee$ implies that $\langle P, E\rangle = 1$. Hence, $P(\{x\})(\hat E(x)-E(x))\leq\langle P,\hat E -E\rangle =\langle P,\hat E\rangle -1\leq 0$, because $\hat E$ is an e-variable and $\hat E-E$ is non-negative. Since $P(\{x\})>0$ and $\hat E\succeq E$, it must be that $\hat E(x)=E(x)$. So, $E$ is maximal.
\end{proof}

\section{Domains with finite cardinality}\label{sec:finite}
For this section, we let $\X$ have finite cardinality $d$. Therefore, we can see probability measures $P$ on $\X$ and scalar functions $\phi$ on $\X$ as vectors in $\R^d$, so that $\langle P, \phi\rangle$ is now simply the dot product in $\R^d$. With this in mind, a hypothesis $\Hh$ on $\X$ is a subset of $\R^d$, and we can define $\Hh^\perp$ as the usual orthogonal set in $\R^d$. Moreover, we remark that, for any $\Phi:\X\to\R^m$, $\PR_\Phi=\PR_\X$.

\begin{lemma}\label{lemma:finspan}
    Let $\X$ have finite cardinality. Let $\Hh$ be a properly constrained hypothesis on $\X$ and $\Phi$ a constraint. Then, $\Span\Phi = \Hh^\perp$.
\end{lemma}
\begin{proof}
    First, note that $\Hh = (\Span\Phi)^\perp\cap\PR_\X$. By \Cref{lemma:profull} and \Cref{lemma:relintsupp}, $\Hh\cap\relint\PR_\X\neq\varnothing$. In particular, by \Cref{lemma:span}, $\lin\Hh = \lin((\Span\Phi)^\perp\cap\PR_\X) = (\Span\Phi)^\perp$. As we are dealing with finite dimensional spaces, $\Span\Phi =(\Span\Phi)^{\perp\perp} = (\lin\Hh)^\perp = \Hh^\perp$.
\end{proof}

\begin{proposition}\label{prop:findualopt}
    If $\Hh$ is a properly constrained hypothesis on a finite set $\X$, $\Ee_\Hh^\vee$ is optimal.
\end{proposition}
\begin{proof}
    \Cref{lemma:finicovpro} and \Cref{lemma:finopt} ensure that the optimal e-class exists. Hence, by \Cref{lemma:maxopt}, it is sufficient to show that $\Ee_\Hh^\vee$ coincides with the set of all maximal e-variables. By \Cref{lemma:dualmax}, all the elements of $\Ee_\Hh^\vee$ are maximal. Therefore, we are left to show that every maximal e-variable is in $\Ee_\Hh^\vee$. Fix a maximal e-variable $E$. We adopt the following proof strategy. First, we show that there is a fully supported $\tilde P\in\Hh$ such that $\langle\tilde P, E\rangle =1$. Using that such $\tilde P$ is in the relative interior of $\Hh$, we establish that $\langle P, E\rangle =1$ for all $P\in\Hh$. Finally, we see that this property is enough to ensure that $E\in\Ee_\Hh^\vee$. 

    We start by noticing that since $\X$ has finite cardinality, $\Hh$ is a convex and closed polytope (it is the intersection of a finite dimensional closed simplex with an affine space). Letting $\V\subseteq\Hh$ denote the (finite) set of vertices of $\Hh$, every $P\in\Hh$ is a convex combination of elements of $\V$.
    
    Fix a maximal e-variable $E$,\footnote{Note that a maximal e-variable has to exist by \Cref{lemma:maxopt}, since we already know that the optimal e-class exists. A more direct argument: the constant function $1$ belongs to $\Ee_\Hh^\vee$, by \Cref{lemma:dualmax} it is a maximal e-variable.} and let $\V_0 = \{\hat P \in\V\,:\,\langle \hat P, E\rangle =1\}$. $\V_0$ is non-empty. Indeed, $E$ cannot be identically null, so \Cref{lemma:finicovpro} implies that $C=\sup_{P\in\Hh}\langle P, E\rangle >0$. Hence, $E/C$ is still an e-variable. Since $C\leq 1$, the maximality of $E$ yields $C=1$. Therefore, $1 =\sup_{P\in\Hh}\langle P, E\rangle = \max_{\hat P\in\V}\langle \hat P,E\rangle$, since every $P\in\Hh$ is a convex combination of vertices of $\Hh$. Hence, $\V_0\neq\varnothing$.

    Let $S =\bigcup_{\hat P\in \V_0}\supp \hat P$. Now, we will show that $S = \X$. By contradiction, let us assume that this is not the case, and fix $x\in\X\setminus S$. Let $\one{x}$ denote the function that is $1$ on $x$ and $0$ everywhere else. If $S\neq\X$, then $\V_0\neq\V$. So, we can define the mapping $\varphi_x:[0,+\infty)\to\R$ as
    $$\varphi_x\;:\;\ee\mapsto \max_{\hat P\in\V\setminus \V_0}\langle \hat P, E+\ee\one{x}\rangle\,.$$
    Since this is the maximum among a finite number of continuous functions, it is continuous. For $\ee$ large enough, $\varphi_x(\ee)>1$. Since $\varphi_x(0)<1$, there is $\tilde\ee>0$ such that $\varphi_x(\tilde\ee)=1$. Let $\tilde E = E + \tilde\ee\one{x}$. For any $\hat P\in \V_0$ we have $\langle \hat P,\tilde E\rangle = \langle\hat P, E\rangle = 1$, as $x\notin\supp(\hat P)$. Moreover, for all $\hat P\in\V\setminus \V_0$, $\langle \hat P, \tilde E\rangle \leq 1$, by definition of $\varphi_x$ and $\tilde\ee$. Therefore, $\tilde E$ is an e-variable, since its expectation is at most one under all the vertices of $\Hh$. As $\tilde\ee>0$, $\tilde E\succ E$, and so $E$ cannot be maximal, a contradiction. We conclude that $S = \X$. 

    From what we have shown so far, there is a fully supported $P^\star\in\Hh$ such that $\langle P^\star,E\rangle = 1$. Indeed, it is enough to take $P^\star = N^{-1}\sum_{\hat P\in\V_0}\hat P$, where $N$ is the cardinality of $\V_0$. In particular, by \Cref{lemma:relintsupp}, $P^\star\in\relint\PR_\X$. Since $\Hh = \PR_\X\cap(\Span\Phi)^\perp$, and $\relint(\PR_\X)\cap(\Span\Phi)^\perp$ is non-empty ($P^\star$ belongs to it), $\relint\Hh = \relint(\PR_\X)\cap(\Span\Phi)^\perp$ (see, e.g., Proposition 2.1.10 in Chapter A of \citealp{hiriart2004fundamentals}). So, $P^\star\in\relint\Hh$. As $\Hh$ is a convex polytope, there are strictly positive coefficients $(\alpha_{\hat P})_{\hat P\in\V}$, summing to $1$, such that $P^\star = \sum_{\hat P\in\V}\alpha_{\hat P}\hat P$ (see, e.g., Remark 2.1.4 in Chapter A of \citealp{hiriart2004fundamentals}). In particular,
    $$0=\langle P^\star, E\rangle-1 = \sum_{\hat P\in\V}\alpha_{\hat P}(\langle \hat P, E\rangle - 1)\,,$$ where $\alpha_{\hat P}>0$ for all $\hat P\in\V$, and so $\langle \hat P, E\rangle = 1$ for all $\hat P\in\V$. Hence, $\langle P, E\rangle = 1$ for all $P\in\Hh$.

    Now, let $\Phi$ be a constraint for $\Hh$. $\langle P, E\rangle =1$ for all $P\in\Hh$, so $E-1\in\Hh^\perp = \Span\Phi$ by \Cref{lemma:finspan}. Since $E$ is non-negative, $E=1+\lambda\cdot\Phi$ for some $\lambda\in\Lambda_\Phi$, and so $E\in\Ee_\Hh^\vee$.
\end{proof}

\section{Matching sets}\label{sec:matching}

Thus far, we have shown the optimality of the dual e-class when $\X$ has finite cardinality. The notions of compatible and matching sets that we now introduce will allow us to build on these finite-cardinality results to address the case of a generic closed set $\X\subseteq\R^n$. 

For any $S\subseteq\X$, recall that $\PR_S = \{P\in\PR_\X\,:\,\supp P\subseteq S\}$. For a measurable $\Phi:\X\to\R^m$, let $\PR_{\Phi|S} = \{P\in\PR_S\,:\,\langle P,\|\Phi\|_1\rangle<+\infty\}$. We notice that $\PR_{\Phi|S}= \PR_\Phi\cap\PR_S$.

\begin{definition}[Compatible and matching sets]
    Let $\Hh$ be a hypothesis on $\X$ and $S\subseteq\X$ a closed set. We say that $S$ is a \emph{compatible} set (for $\Hh$) if there is $P\in\Hh$ such that $\supp P\subseteq S$. We say that $S$ is a \emph{matching} set (for $\Hh$) if there is $P\in\Hh$ such that $\supp P=S$.
\end{definition}

\begin{definition}[Restriction of a hypothesis]
    The restriction of a hypothesis $\Hh$ on $\X$ to a compatible set $S\subseteq\X$ is defined as $$\Hh_S = \Hh\cap\PR_S = \{P\in\Hh\;:\;\supp P\subseteq S\}\,.$$ 
\end{definition}

By definition of compatible set, $\Hh_S$ is non-empty, and hence a hypothesis on $S$. We list here a few results outlining basic properties of the restrictions of finitely and properly constrained hypotheses that we will use in the next section. The proofs, rather short, can be found in \Cref{app:proofs}.

\begin{lemma}\label{lemma:Hsfic}
    Let $\Hh$ be a finitely constrained hypothesis on $\X$ and $\Phi$ a constraint. A closed set $S\subseteq\X$ is compatible if, and only if, $0\in\conv\Phi(S)$. In such case, $\Hh_S$ is a finitely constrained hypothesis on $S$ and $\Phi\big|_S$ is a constraint for it. In particular, $\Hh_S = \{P\in\PR_{\Phi|S}\,:\,\langle P,\Phi\big|_S\rangle=0\}$.
\end{lemma}
\begin{lemma}\label{lemma:comsupp}
    Let $\Hh$ be a properly constrained hypothesis on $\X$ and $S\subseteq\X$ a compatible set. $S$ is a matching set if, and only if, $\Hh_S$ is a properly constrained hypothesis on $S$.
\end{lemma}

\begin{lemma}\label{lemma:finprocont}
    Let $\Hh$ be a properly constrained hypothesis on $\X$. A finite union of matching sets is a matching set. In particular, every finite subset of $\X$ is contained in a finite matching set.
\end{lemma}

As a first application of what we have just introduced, we can now prove \Cref{prop:genspan}.

\begin{manual}{\Cref{prop:genspan}}
    Let $\Hh$ be a properly constrained hypothesis on $\X$, with $\Phi:\X\to\R^m$ and $\Phi':\X\to\R^{m'}$ two constraints. Then, $\Span\Phi = \Span\Phi'$. 
\end{manual}
\begin{proof}
    If $\X$ is finite, the result is a direct consequence of \Cref{lemma:finspan}. So, consider the case where $\X$ has infinitely many elements. Clearly, it is enough to show that $\Span\Phi'\subseteq\Span\Phi$, or equivalently that, for any $\phi\in\Span\Phi'$, $\Span\Phi\cup\{\phi\}$ is not a linearly independent family. 
    
    Fix $\phi\in\Span\Phi'$. Let $d$ be the dimension of $\Span\Phi$. By \Cref{lemma:ludovic}, it is enough to show that for any $S\subseteq\X$ with $d+1$ elements, $V_S = \lin(\Span\Phi\big|_S\cup\{\phi\big|_S\})$ has dimension at most $d$. So, fix any $S\subseteq\X$ with $d+1$ elements. By \Cref{lemma:finprocont}, there is a finite matching set $S'\supseteq S$. Since $\Hh_{S'}$ is a properly constrained hypothesis on $S'$ (\Cref{lemma:comsupp}) and both $\Phi\big|_{S'}$ and $\Phi'\big|_{S'}$ are constraints for it (\Cref{lemma:Hsfic}), we have that $\Span\Phi\big|_{S'} = \Span\Phi'\big|_{S'}$ by \Cref{lemma:finspan}, because $S'$ has finite cardinality. As $S\subseteq S'$, $\phi\big|_{S}\in\Span\Phi\big|_{S}$. So, $\dim V_S= \dim (\Span\Phi\big|_{S})\leq \dim (\Span\Phi) = d$.
\end{proof}

\section{Optimality of the dual e-class}\label{sec:optimal}
At the end of this section we finally prove \Cref{thm:dualopt}. First, we will establish the result for the case where $\X$ is compact, then extend it to the general case of a closed $\X\subseteq\R^n$. We introduce the following notation. For any $\Phi:\X\to\R^m$ and any $S\subseteq\X$, we define the set $$\Lambda_{\Phi|S} = \left\{\lambda\in\R^m\;:\;\sup_{x\in S}\lambda\cdot\Phi(x)\leq 1\right\}\,.$$ We state here two elementary properties of these sets. 

\begin{lemma}\label{lemma:intlbd}
    For any sequence $(S_i)_{i\geq 1}$ of subsets of $\X$, $\bigcap_{i=1}^\infty \Lambda_{\Phi|S_i} = \Lambda_{\Phi|S} $, with $S = \bigcup_{i=1}^\infty S_i$. 
\end{lemma}
\begin{proof}
    For all $i\geq1$, $S\supseteq S_i$, so $\Lambda_{\Phi|S}\subseteq\bigcap_{i=1}^\infty\Lambda_{\Phi|S_i}$. Moreover, if $\lambda\in\bigcap_{i=1}^\infty \Lambda_{\Phi|S_i}$, then we have $ \lambda\cdot\Phi(x)\leq1$ for every $x\in S_i$, for any $i\geq 1$. Therefore, $\lambda\cdot\Phi(x)\leq 1$ for all $x\in S$, so $\lambda\in\Lambda_{\Phi|S}$.
\end{proof}

\begin{lemma}\label{lemma:LFclo}
    If $\Phi$ is continue, $\Lambda_{\Phi|S}$ is closed, for any $S\subseteq\X$. 
\end{lemma}
\begin{proof}
    It follows from the lower semi-continuity of the supremum of continuous functions.
\end{proof}
From \Cref{lemma:LFclo}, if $\Phi$ is a constraint of a finitely generated hypothesis, $\Lambda_{\Phi}$ is closed. If $\Phi$ is a minimal constraint of a properly generated hypothesis, a stronger result holds: $\Lambda_\Phi$ is compact.
\begin{lemma}\label{lemma:mincomp}   
    Let $\Hh$ be a properly constrained hypothesis on $\X$ and $\Phi$ a constraint. Then $\Phi$ is minimal, if, and only if, $\Lambda_\Phi$ is compact. More precisely, $\Phi$ is minimal if, and only if, there is a finite matching set $S^\star$ such that $\Lambda_{\Phi|S^\star}$ is compact. If such $S^\star$ exists, $\Lambda_{\Phi|S}$ is compact for any $S\subseteq\X$ that contains $S^\star$, and any closed set $S\subseteq \X$ that contains $S^\star$ is a matching set.
\end{lemma}
The proof of \Cref{lemma:mincomp} (detailed in \Cref{app:proofs}) builds on the characterisation of minimal constraints that we stated in \Cref{lemma:prominiff}. An immediate corollary of the above is the following.
\begin{corollary}\label{cor:seccomp}
    If $\Hh$ is a properly constrained hypothesis on $\X$, 
    any sequence $(E_i)_{i\geq 1}\subseteq\Ee_\Hh^\vee$ has a subsequence that converges point-wise on $\X$ to some $E\in\Ee_\Hh^\vee$.
\end{corollary}
\begin{proof}
    By \Cref{lemma:prominc} we can fix a minimal constraint $\Phi$ for $\Hh$. By \Cref{lemma:mincomp}, $\Lambda_\Phi$ is compact. For any $(E_i)_{i\geq 1}\subseteq\Ee_\Hh^\vee$, we consider a corresponding sequence $(\lambda_i)_{i\geq 1}\subseteq\Lambda_\Phi$, such that $E_i = 1-\lambda_i\cdot\Phi$ for all $i\geq1$. The existence of a subsequence of $(\lambda_i)_{i\geq 1}$ converging to some $\lambda^\star\in\Lambda_\Phi$ implies that there is a subsequence of $(E_i)_{i\geq1}$ that converges point-wise to $1-\lambda^\star\cdot\Phi$, which is in $\Ee_\Hh^\vee$.
\end{proof}
Before proving the optimality of the dual e-class, we establish one last preliminary result.
\begin{lemma}\label{lemma:Dilemma}
    Let $\Hh$ be a properly constrained hypothesis on $\X$. Fix any countable set $S\subseteq \X$. Then, for each $E\in\Ee_\Hh$ we can find $E'\in\Ee_{\Hh}^\vee$ such that $E'(x)\geq E(x)$ for every $x\in S$.
\end{lemma}
\begin{proof}
    By \Cref{lemma:prominc} we can fix a minimal constraint $\Phi$ for $\Hh$. By \Cref{lemma:mincomp}, there is a finite matching set $S'_0$ such that $\Lambda_{\Phi|S'_0}$ is compact. Consider a sequence $(x_i)_{i\geq 1}\subseteq\X$, dense in $\X$ and containing $S$. By \Cref{lemma:finprocont}, for all $i\geq1$ we have a finite matching set $S'_i$ containing $x_i$. For each $i\geq 1$, let $S_i = \bigcup_{j=0}^iS'_i$. These sets are finite matching sets by \Cref{lemma:finprocont}. Let $D = \bigcup_{i=1}^\infty S_i$. 
    
    Fix $E\in\Ee_\Hh$. 
    For all $i\geq 1$, $\Hh_{S_i}$ is a properly constrained on $S_i$ (\Cref{lemma:comsupp}). By \Cref{prop:findualopt} there is $\lambda_i\in\Lambda_{\Phi|S_i}$ such that $E_{\lambda_i}(x)\geq E(x)$, for all $x\in S_i$, where we let $E_\lambda = 1-\lambda\cdot\Phi$ for $\lambda\in\R^m$. Since each $\lambda_i$ is in $\Lambda_{\Phi|S_0'}$, which is compact by construction, $(\lambda_i)_{i\geq 1}$ has a subsequence $(\lambda^\star_i)_{i\geq 1}$ converging to some $\lambda^\star\in\Lambda_{\Phi|S_0'}$. Note that $(S_i)_{i\geq 1}$ is non-decreasing under set inclusion, namely $S_i\subseteq S_{i+1}$ for all $i\geq1$. In particular, we can consider the subsequence $(S_i^\star)_{i\geq 1}$ of $(S_i)_{i\geq 1}$, obtained by selecting those indices that had been used to generate $(\lambda^\star_i)_{i\geq 1}$, and still $D = \bigcup_{i=1}^\infty S^\star_i$. For any $i\geq 1$ and $j\geq i$, $\lambda^\star_j\in\Lambda_{\Phi|S^\star_i}$, so $\lambda^\star\in\bigcap_{i=1}^\infty\Lambda_{\Phi|S^\star_{i}}=\Lambda_{\Phi|D}$, by \Cref{lemma:intlbd}. Clearly, the convergence of $(\lambda^\star_{i})_{i\geq1}$ to $\lambda^\star$ implies that $(\lambda_{i}^\star\cdot \Phi)_{i\geq1}$ converges point-wise to $\lambda^\star\cdot\Phi$ on $\X$. Fix $k\geq 1$. For any $i\geq k$, $E_{\lambda^\star_{i}}(x)\geq E(x)$ if $x\in S^\star_{k}$, since $S^\star_{k}\subseteq S^\star_{i}$. So, for all $x\in S^\star_k$, $E_{\lambda^\star}(x) = \lim_{i\to\infty} E_{\lambda^\star_i}(x) \geq E(x)$. This holds for any $k\geq 1$, so $E_{\lambda^\star}(x)\geq E(x)$ for all $x\in S\subseteq D$. 

    To conclude we only need to show that $E_{\lambda^\star}\in\Ee_\Hh^\vee$. Since $\lambda^\star\in\Lambda_{\Phi|D}$, we have that $\lambda^\star\cdot\Phi(x)\leq 1$ for all $x\in D$. As $D\supseteq(x_i)_{i\geq1}$ is dense in $\X$ and $\lambda^\star\cdot\Phi$ is continuous, we deduce that $\lambda^\star\in\Lambda_\Phi$.
\end{proof}

\begin{proposition}\label{prop:compdualopt}
    Assume that $\X\subseteq\R^n$ is compact. Let $\Hh$ be a properly constrained hypothesis on $\X$. Then, the optimal e-class exists and coincides with the dual e-class $\Ee_\Hh^\vee$.
\end{proposition}
\begin{proof}
    Fix $E\in\Ee_\Hh$. Only one of the following mutually contradictory statements is true.
    \begin{enumerate}\setlength{\itemsep}{0pt}
        \item There is $\ee>0$ such that, for every $\hat E\in\Ee_\Hh^\vee$, there is $x\in\X$ such that $E(x)>\hat E(x)+\ee$.
        \item For every $\ee>0$ there is $\hat E\in\Ee_\Hh^\vee$ such that, for all $x\in\X$, $E(x)\leq \hat E(x)+\ee$. 
    \end{enumerate}
    Let us show that the claim (1) is false. First, note that since the space of continuous functions on a compact set in $\R^n$ is a separable metric space  (under the uniform norm), there is a sequence $(\hat E_i)_{i\geq 1} \subseteq\Ee_\Hh^\vee$ that is dense in $\Ee_\Hh^\vee$ (under the uniform convergence). If option (1) holds, there is $\ee>0$ and a sequence $(x_i)_{i\geq 1}\subseteq\X$ such that, for each $i\geq 1$,  $E(x_i)> \hat E_i(x_i)+\ee$. By \Cref{lemma:Dilemma}, there is $\hat E\in\Ee_\Hh^\vee$ such that $\hat E(x_i)\geq E(x_i)> \hat E_i(x_i)+\ee$, for all $i\geq 1$. Since $(\hat E_i)_{i\geq1}$ is dense in $\Ee_\Hh^\vee$, there is an index $i^\star$ such that $\sup_{x\in\X}|\hat E(x) - \hat E_{i^\star}(x)|<\ee$, a contradiction. So, (1) is false.

    As the claim (2) must be true, we can find a sequence $(E_i)_{i\geq 1}\subseteq\Ee_\Hh^\vee$, such that $E\preceq E_i+i^{-1}$ for all $i\geq 1$. By \Cref{cor:seccomp}, there is a subsequence of $(E_i)_{i\geq1}$ that converges point-wise to some $E^\star\in\Ee_\Hh^\vee$. It follows that $E\preceq E^\star$, from which we deduce that $\Ee_\Hh^\vee$ is a majorising e-class. As all its elements are maximal (\Cref{lemma:dualmax}), it is the optimal e-class by \Cref{lemma:maxopt}.
\end{proof}

\begin{manual}{\Cref{thm:dualopt}}
    Let $\Hh$ be a properly constrained hypothesis. Then, the dual e-class $\Ee_\Hh^\vee$ is optimal.
\end{manual}
\begin{proof}
    As all elements of $\Ee_\Hh^\vee$ are maximal (\Cref{lemma:dualmax}), by \Cref{lemma:maxopt} it is enough to show that $\Ee_\Hh^\vee$ is a majorising e-class. Fix $E\in\Ee_\Hh$. By \Cref{lemma:prominc} we can fix a minimal constraint $\Phi$ for $\Hh$. By \Cref{lemma:mincomp}, there is a finite matching set $S^\star\subseteq\X$ such that any closed set $S\subseteq\X$ containing $S^\star$ is a matching set inducing a compact $\Lambda_{\Phi|S}$. As $\X$ is closed, it is sigma-compact, and there is a sequence $(K_i)_{i\geq 1}$ of compact matching sets, such that $S^\star\subseteq K_i\subseteq K_{i+1}$, for all $i\geq 1$, and $\X = \bigcup_{i=1}^\infty K_i$. By construction, $\Hh_{K_i}$ is properly constrained (\Cref{lemma:comsupp}) and $\Phi\big|_{K_i}$ a constraint for it (\Cref{lemma:Hsfic}). From \Cref{prop:compdualopt}, for each $i\geq1$ there is $\lambda_i\in\Lambda_{\Phi|K_i}$ such that $E_{\lambda_i}(x)\geq E(x)$ for all $x\in K_i$, where as usual $E_{\lambda} = 1-\lambda\cdot\Phi$. For all $i\geq1$, because $K_i\subseteq K_{i+1}$ we have $\Lambda_{\Phi|K_i}\supseteq\Lambda_{\Phi|K_{i+1}}$. In particular, for any $j\geq 1$, $(\lambda_i)_{i\geq j}$ is contained in $\Lambda_{\Phi|K_j}$, which is compact by \Cref{lemma:mincomp}. So, there is a subsequence $(\lambda'_i)_{i\geq 1}$ of $(\lambda_i)_{i\geq 1}$ converging to $\lambda^\star\in\bigcap_{j\geq1}\Lambda_{\Phi|K_j}$. Let $(K'_i)_{i\geq 1}$ be the corresponding subsequence of $(K_i)_{i\geq 1}$, obtained selecting the same indices. By \Cref{lemma:intlbd}, $\lambda^\star\in\Lambda_\Phi$, so  $E_{\lambda^\star}\in\Ee_\Hh^\vee$. Fix any $x\in\X$. There is an index $i^\star\geq 1$ such that $x\in K_i'$ for all $i\geq i^\star$. Thus, $E_{\lambda'_i}(x)\geq E(x)$ for all $i\geq i^\star$, and $E_{\lambda^\star}(x)\geq E(x)$. As the choice of $x$ has been arbitrary, $E_{\lambda^\star}\succeq E$. Therefore, $\Ee_\Hh^\vee$ is a majorising e-class.
\end{proof}

\section{Extensions}\label{sec:extensions}

\subsection{Testing finitely (non-properly) constrained hypotheses}\label{sec:finitely}
So far we have mostly focused on those finitely constrained hypotheses $\Hh$ that are also properly constrained. Indeed, if $\Hh$ is not properly constrained, by \Cref{lemma:finicovpro} and \Cref{cor:nonex} the optimal e-class does not exist. However, we can still characterise an ``optimal'' e-variable testing procedures. The key idea is the following result, whose proof is detailed in \Cref{app:proofs}.
\begin{lemma}\label{lemma:X0}
    Let $\Hh$ be a finitely constrained hypothesis on $\X$. Let $\X_0=\bigcup_{P\in\Hh}\supp P$. Then, $\X_0$ is closed and $\Hh$ is a properly constrained hypothesis on $\X_0$.
\end{lemma}
\Cref{lemma:X0} tells us that any finitely hypothesis $\Hh$ on $\X$ can be seen as a properly constrained hypothesis on some smaller set $\X_0$. By the definition of $\X_0$, we know that observing $x_t\notin\X_0$ is incompatible with the hypothesis $\Hh$, as such $x_t$ is not in the support of any $P\in\Hh$. So, we can design the following testing procedure. Assuming that $\Hh$ holds, we see it as a hypothesis on $\X_0$ and we design a testing game on $\X_0$ (as in \Cref{def:game}) restricted to the optimal e-class of $\Hh$ on $\X_0$. At each new observation $x_t$, if $x_t\notin\X_0$ we automatically reject $\Hh$. Conversely, if $x_t\in\X_0$, we let the player in the testing game observe $x_t$, and as usual we reject $\Hh$ if the cumulative reward exceeds the threshold $\log(1/\delta)$, with $\delta\in(0,1)$ the type II  confidence level.

\subsection{Loosely constrained hypotheses}

Finitely constrained hypotheses involve the equality $\langle P,\Phi\rangle=0$. Here, we relax this condition.

\begin{definition}[Loosely constrained hypotheses]\label{def:loosely}
    A hypothesis $\Hh$ on $\X$ is \emph{loosely constrained} if there are two continuous functions $\Phi':\X\to\R^{m'}$ and $\Phi'':\X\to\R^{m''}$ such that, denoted as $\Phi:\X\to\R^{m'+m''}$ the mapping $x\mapsto\Phi(x) = (\Phi'(x),\Phi''(x))$, we have that $0\in\conv\Phi(\X)$ and
    $$\Hh = \{P\in\PR_{\Phi'}\cap\PR_{\Phi''}\;:\;\langle P, \Phi'\rangle =0\,,\;\langle P,\Phi''\rangle\geq0\}\,,$$ where the inequality holds component-wise. We call $\Phi'$ a \emph{tight} constraint and $\Phi''$ a \emph{slack} constraint. If moreover $0\in\relint(\conv\Phi(\X))$, $\Hh$ is said to be loosely \emph{properly} constrained.
\end{definition}
For any properly loosely constrained hypothesis the optimal e-class exists. This is made explicit in the next proposition (see \Cref{app:loose} for the proof).

\begin{proposition}\label{prop:loosely}
    Let $\Hh$ be a properly loosely constrained hypothesis on $\X$, with $\Phi':\X\to\R^{m'}$ and $\Phi'':\X\to\R^{m''}$ tight and slack constraints.
    Define $$\tilde\Lambda_{\Phi',\Phi''} = \big\{(\lambda',\lambda'')\;:\;\lambda'\in\R^{m'}\,,\;\lambda''\in[0,+\infty)^{m''}\,,\;1-\lambda'\cdot\Phi'-\lambda''\cdot\Phi''\succeq 0\big\}\,.$$ Then, the optimal e-class for $\Hh$ exists and coincides with $$\tilde\Ee_{\Hh}^\vee = \big\{1-\lambda'\cdot\Phi'-\lambda''\cdot\Phi''\;:\;(\lambda',\lambda'')\in\tilde\Lambda_{\Phi',\Phi''}\big\}\,.$$
\end{proposition}

The set $\Hh = \{P\in\PR_\Phi\,:\,\langle P, \Phi\rangle=0\}$, where $\Phi:\X\to\R^m$ is a continuous function, is non-empty (and thus a hypothesis) if, and only if, $0\in\conv\Phi(\X)$ (\Cref{lemma:car} and \Cref{lemma:convP}). Conversely, when we allow for inequality constraints, as in the definition of loosely constrained hypotheses, this is not anymore a necessary condition. More explicitly, $\Hh = \{P\in\PR_{\Phi'}\cap\PR_{\Phi''}\,:\,\langle P,\Phi'\rangle =0\,,\,\langle P,\Phi''\rangle\geq 0\}$ can be non empty if $0\notin\conv\Phi(\X)$, (where $\Phi'$ and $\Phi''$ are continuous functions on $\X$ and $\Phi=(\Phi',\Phi'')$). However, here we only focus on the case $0\in\conv\Phi(\X)$, as required by \Cref{def:loosely}, and defer to future work the analysis of the setting of generic hypotheses defined via loose constraints. We also note that \Cref{lemma:X0} cannot be trivially extended to the case of loosely constrained hypotheses, and it is not always possible to restrict the domain $\X$ of a loosely constrained hypothesis to make it properly loosely constrained. For an example where this is not possible, consider the case of a loosely constrained hypothesis on $\X=\R$, with tight constraint $\Phi':x\mapsto 1-x$ and slack constraint $\Phi'':x\mapsto x^2-1$.

\section{Algorithmic mean estimation}\label{sec:AME}

In \cite{clerico2024optimality}, the notion of optimal e-class was developed in the context of applying the testing framework to mean estimation for random variables bounded in $[0,1]$. The class of constrained hypotheses that we have introduced allows for a substantial extension of those results.
\subsection{Mean estimation for bounded random variables}\label{sec:meanbounded}
Let $\X$ be a compact set in $\R^n$ and let $\Cc$ denote its convex hull. We observe independent draws $(X_i)_{i\geq 1}$ from an unknown distribution $P^\star\in\PR_\X$, whose mean $\mu^\star\in\Cc$ needs to be estimated. It is possible to generate an anytime-valid sequence of confidence sets (usually referred to as \emph{confidence sequence}) via sequential hypothesis testing (see, e.g., \citealp{ramdas2024hypothesis}). More precisely, we are looking for a sequence of random sets $(S_t)_{t\geq 1}$, adapted to $(X_i)_{i\geq 1}$ and such that 
$$\Pp\big(\mu^\star\in S_t\,,\;\forall t\geq 1\big)\geq 1-\delta\,,$$
where $\Pp$ denotes the law of the sequence of observations and $\delta\in(0,1)$ is a fixed confidence level.

Let $\PR_1$ denote the set of measures $P\in\PR_\X$ satisfying $\langle P, \|X\|_1\rangle <+\infty$. Clearly, for $P\in\PR_1$ the mean $\langle P, X\rangle$ is well defined. For $\mu\in\Cc$, let $$\Hh_\mu = \left\{P\in\PR_1\;:\;\langle P, X\rangle =\mu\right\}\,.$$ $\Hh_\mu$ is a hypothesis on $\X$ (it is non-empty by \Cref{lemma:car}). It is finitely constrained, with $\Phi_\mu:x\mapsto \mu-x$ a constraint (as $0\in\conv\Phi_\mu(\X)$ because $\mu\in\Cc$). Moreover, if $\mu\in\relint\Cc$, $\Hh_\mu$ is a properly constrained hypothesis on $\X$. Hence, for the sake of simplicity, we assume as known that $\mu^\star\in\relint\Cc$. Therefore, we are only interested in confidence sets $S_t\subseteq\relint\Cc$. We note, however, that the general case can be addressed by dealing with the values of $\mu$ on the relative boundary of $\Cc$ as discussed in \Cref{sec:finitely}. A valid confidence sequence for $\mu^\star\in\relint\Cc$ with confidence level $\delta\in(0,1)$ is $(S_t)_{t\geq 1}$, where $$S_t = \big\{\mu\in\relint\Cc\;:\;R_t(\mu)\leq\log(1/\delta)\big\}\,.$$
Here, $R_t(\mu)$ denotes the cumulative reward, up to round $t$, of a player who is playing a testing game as in \Cref{def:game} on the null hypothesis $P^\star\in\Hh_\mu$, observing the sequence $(X_i)_{i\geq 1}$. From the fact that $\{\mu^\star = S_t\,,\,\forall t\geq 1\}\iff\{R_t(\mu^\star)\leq\log(1/\delta)\,,\,\forall t\geq 1\}$, which happens with probability at least $1-\delta$ by \Cref{prop:game}, it follows that $(S_t)_{t\geq 1}$ is a confidence sequence.

Of course, in order to get a tight confidence sequence, we shall restrict the testing games to the optimal e-classes (see also the discussion in \citealp{clerico2024optimality}). For every $\mu\in\relint\Cc$, $\Hh_\mu$ is a properly constrained hypothesis and we can leverage the theory that we have developed so far. In particular, we have $\Lambda_{\Phi_\mu} = \{\lambda\in\R^n\,:\,\sup_{x\in\X}\lambda\cdot (\mu-x)\leq 1\}$, and the optimal e-class is
$$\Ee_{\Hh_\mu}^\vee=\left\{x\mapsto 1+\lambda\cdot(x-\mu)\,:\,\lambda\in\Lambda_{\Phi_\mu}\right\}\,.$$
When $\X=[0,1]$, $\Ee_{\Hh_\mu}^\vee$ is the \emph{coin-betting} e-class studied in \cite{clerico2024optimality}, which had been used for mean estimation by \cite{orabona2023tight} and \cite{waudbysmith23estimating}.

\subsection{Mean estimation for heavy-tailed distributions}\label{sec:heavy}
Consider the case $\X=\R$. We want to estimate the mean $\mu^\star$ of an unknown probability distribution $P^\star\in\PR_\R$, under the assumption that $\langle P^\star, X^2\rangle\leq 1$. Note that this is a heavy-tail mean estimation problem, as we only have an upper bound on a moment of $P^\star$. Let $(X_i)_{i\geq 1}$ be a sequence of independent draws from $P^\star$. As we have just done in \Cref{sec:meanbounded}, we aim at obtaining a confidence sequence $(S_t)_{t\geq1}$ for $\mu^\star$. Clearly, the fact that $\langle P^\star, X^2\rangle\leq 1$ implies that $\mu^\star\in[-1,1]$. For $\mu\in[-1,1]$, let 
$$\Hh_\mu = \big\{P\in\PR_\R\;:\;\langle P, X^2\rangle\leq 1\,,\;\langle P, X\rangle=\mu\big\}\,,$$ where we note that the fact that $\langle P, X^2\rangle\leq 1$ automatically implies that $\langle P, X\rangle$ is well defined. Clearly, $\Hh_1 = \{\delta_1\}$ and $\Hh_{-1}=\{\delta_{-1}\}$. It is easy to check that, for $\mu\in(-1,1)$, $\Hh_\mu$ is a loosely properly constrained hypothesis on $\R$, where $\Phi'_\mu:x\mapsto\mu-x$ is a tight constraint and $\Phi'':x\mapsto 1-x^2$ is a slack constraint. Denoting as $R_t(\mu)$ the cumulative reward (up to round $t$) of a player betting in a testing game on $\Hh_\mu$ and observing the sequence $(X_i)_{i\geq 1}$, reasoning as in \Cref{sec:meanbounded}, we deduce that 
$$S_t = \big\{\mu\in(-1,1)\;:\;R_t(\mu)\leq\log(1/\delta)\big\} \cup U_t$$
defines a confidence sequence for $\mu^\star$, where $U_t = \{1\}$ if for every $s\leq t$ all $X_s=1$, $U_t = \{-1\}$ if for every $s\leq t$ all $X_s=-1$, and $U_t=\varnothing$ otherwise.

By \Cref{prop:loosely}, for $\mu\in(-1,1)$ the optimal e-class is
$$\tilde\Ee_{\Hh_\mu}^\vee = \big\{x\mapsto 1+\alpha(x-\mu) + \beta(x^2-1)\;:\;(\alpha,\beta)\in\tilde\Lambda_{\Phi'_\mu,\Phi''}\}\,,$$
where $\tilde\Lambda_{\Phi'_\mu,\Phi''} = \big\{(\alpha,\beta)\in\R^2
\,:\,\alpha^2+4\mu\alpha\beta+4\beta^2-4\beta\leq0\}$
is the convex hull of an ellipse lying on the half-plane where $\beta\geq 0$. 

This result can be extended to the more general case $\X = \R^n$, with the bounded second moment assumption replaced by $\langle P^\star, \|X\|^{1+\ee}\rangle\leq B$, for some given real constants $B>0$ and $\ee>0$, and where $\|\cdot\|$ denotes the standard Euclidean norm. We can also consider the case of a bounded central moment, namely $\langle P^\star,\|X-\mu^\star\|^{1+\ee}\rangle\leq B$. Indeed, in such case, $\Phi'_\mu:x\mapsto\mu -x$ and $\Phi''_\mu\mapsto B-\|x-\mu\|^{1+\ee}$ are the tight and slack constraints of a loosely properly constrained hypothesis for any $\mu\in\R^n$.

The problem of finding confidence sequences for heavy-tailed random variables via betting strategies has been addressed in the literature. \cite{agrawal2021regret} obtained anytime-valid lower bounds for real random variables with bounded $(1+\ee)$-th (non-centred) moment, via an approach that implicitly uses e-variables. Their Lemma 17 is in fact a regret upper bound for the cumulative reward of a player in a testing game as in \Cref{def:game}. Notably, our results show that their analysis considered games restricted to the optimal e-class. On the other hand, \cite{wang2023catoni} leveraged ideas from \cite{catoni2012challenging} to set up testing games in order to find confidence sequences for the mean of real random variables with bounded $(1+\ee)$-th central moment. The theory that we have developed in this work shows that their choice of restriction for the e-variables does not correspond to the optimal e-class.

\section{Conclusion}\label{sec:conclusion}
This work helps to deepen the understanding of the notion of optimal e-class introduced in \cite{clerico2024optimality} for hypothesis testing via e-variables, focusing on its characterisation for non-parametric hypotheses defined by a finite number of continuous constraints. We now conclude with some final remarks and possible directions for future research. 

First, the key concepts of maximal e-variable and majorising and optimal e-class that we discuss in this work are strongly connected to the classical statistical notions of \emph{admissibility} and \emph{completeness} (see, e.g., Section 1.8 in \citealp{lehmann2005testing}). Indeed, with this terminology, a maximal e-variable can be essentially thought of as an admissible e-variable, in the sense that no other e-variable uniformly dominates it, a majorising e-class is a complete class of e-variables, and the optimal e-class represents the minimal complete class. Notably, the idea of admissibility has already been explored in the e-variable literature. Indeed, \cite{ramdas2022admissible} considered a notion of admissible wealth processes, which in some contexts is closely related to the maximality for e-variables discussed here. However, their results do not suffice to characterise the optimal e-class in the setting that we study, as already pointed out by \cite{clerico2024optimality}. Indeed, it would be possible to use their results if we were considering hypotheses where a common reference measure (namely a measure for which every element in $P$ admits a density) exists, which is not the case in the setting we study. Admissibility has also been recently examined in a different context involving e-variables, when ``merging'' e-variables \citep{wang2024admissible}. Exploring the connections between those ideas and the concepts discussed in this work could be a promising direction for future research.

To establish our results, we employed a standard functional analytic approach. We began by proving the optimality of the dual e-class for a set $\X$ with finite cardinality, then extended it to the case of a compact $\X$ using a dense countable subset, and finally leveraged sigma-compactness for a general closed $\X$. Our proofs are self-contained and avoid relying on advanced tools from functional analysis or measure theory. Building on the formulation and findings presented in the first pre-print of this work, \cite{larsson2025evariables} have recently re-derived our characterization of the optimal e-class for constrained hypotheses and have extended it to a significantly more general framework. To achieve this broader generality, their approach relies on sophisticated topological and functional-analytic techniques to derive duality results in functional spaces. Nonetheless, we believe that in scenarios where the problem can be reduced to finite-dimensional analysis, our more straightforward and elementary method remains valuable and more accessible.

As highlighted in the introduction, our analysis is restricted to the characterisation of optimality within the testing-by-betting framework with e-values, as defined in \Cref{def:game}. More generally, however, sequential testing with e-values may proceed differently. At each round $t$ one might select an e-variable $E(x_1,\dots, x_t)$, with respect to a hypothesis concerning the entire sequence of observations up to time $t$. In such settings, it is natural to define the hypothesis $\Hh$ on the space of stochastic processes taking values in $\X$, rather than on $\X$ itself. This broader perspective is, for instance, essential when testing whether a sequence of observations is i.i.d. These kinds of sequential tests typically rely on e-processes, namely non-negative stochastic processes whose expected value remains bounded by one, under the null hypothesis, for any stopping time \citep{ramdas2024hypothesis, ramdas2022admissible}. In contrast, the present work focuses exclusively on single-round e-variables, where the hypothesis is defined on on $\PR_\X$, the space of distributions over a single observation. This framework can be appropriate when one assumes that the observed data arise from i.i.d.~draws from a $P\in\PR_\X$, and is interested solely in verifying whether this $P$ is within $\Hh$. In this sense, tests that might reject the hypothesis on the basis that the data appear non-i.i.d.~are not considered, as the i.i.d.~assumption is given for granted (one might want to formalise this by asking for test that has power under any alternative stating that the data are i.i.d.). The more general problem of characterising the optimal class of e-processes for suitably constrained hypotheses (for instance under i.i.d. assumptions, or assuming that constraints hold conditionally) is left open for future investigation.

Finally, open questions also remain regarding the minimal conditions, on the hypothesis $\Hh$, necessary for the existence of the optimal e-class, and the extent to which such e-class can be explicitly characterised, when it does exist. 

\subsection*{Acknowledgements}
I would like to thank Gergely Neu, Peter Gr\"unwald, Gabor Lugosi, Nishant Mehta, Hamish Flynn, Ludovic Schwartz, Claudia Chanu, and Riccardo Camerlo for the insightful discussions that contributed to this work, and Nick Koning for pointing out interesting connections with related literature. This project was funded by the European Research Council (ERC), under the European Union’s Horizon 2020 research and innovation programme (grant agreement 950180).\\
Since the initial appearance of this manuscript, parts of it have been simplified thanks to valuable suggestions. In particular, I would like to thank Martin Larsson for proposing a much simpler example demonstrating the non-existence of the optimal e-class, which, unlike the example in the original version, avoids the use of non-measurable sets.
\newpage
\bibliographystyle{abbrvnat}
\bibliography{bib}
\newpage
\appendix
\section{Omitted proofs}\label{app:proofs}

\begin{manual}{\Cref{lemma:countopt}}
    If $\X$ is countable and $\sup_{P\in\Hh}P(\{x\})>0$ for all $x\in\X$, the optimal e-class exists.
\end{manual}
\begin{proof}
    This proof makes use of elementary properties of ordinals. For an accessible reference, see \cite{aliprantis06infinite}.

    By \Cref{lemma:maxopt}, it is enough to show that every e-variable is majorised by a maximal e-variable, which implies that the set of all maximal e-variables is a majorising class. Fix $E_1\in\Ee_\Hh$. If $E_1$ is already maximal, we are done. So, we focus on the case of an $E_1$ that is not maximal.

    Consider the transfinite sequence (namely, net), starting from $E_1$ and indexed in $\Omega_0=[1,\omega_1)$ (the set of the countable ordinals), defined as follows. For any successor ordinal $\beta>1$ in $\Omega_0$, let $E_\beta = E_{\beta-1}$ if $E_{\beta-1}$ is maximal, otherwise pick any $E\in\Ee_\Hh$ such that $E\succ E_{\beta-1}$ and let $E_\beta=E$. For any limiting ordinal $\gamma\in\Omega_0$, let $(\beta_n)_{n\geq 1}\subseteq\Omega_0$ be an increasing sequence converging to $\gamma$ (under the order topology). For $x\in\X$, let $A(x)=1/\sup_{P\in\Hh}P(\{x\})\in[1,+\infty)$. Note that for any $E\in\Ee_\Hh$ and any $x\in\X$, $E(x)= A(x)E(x)\sup_{P\in\Hh}P(\{x\})\leq A(x)\sup_{P\in\Hh}\langle P, E\rangle \leq A(x)$. So, for each $x\in\X$, the sequence $(E_{\beta_n}(x))_{n\geq 1}$ is non-decreasing and bounded in $[0, A(x)]$, and hence admits a finite limit. Therefore, we can  define the point-wise limit $E_\gamma = \lim_{n\to\infty}E_{\beta_n}$. Moreover, by Beppo Levi theorem, $E_\gamma$ is Borel and $\langle P, E_\gamma\rangle \leq 1$, for any $P\in\Hh$. So, $E_\gamma\in\Ee_\Hh$.
    
    Now, consider the set $B\subseteq\Omega_0$ of the indices where the transfinite sequence has jumps (namely, $E_{\beta}\succ E_{\beta-1}$ if, and only if, $\beta\in B$). Since $E_1$ is not maximal, $B$ is non-empty. For each $\beta\in B$ we can select a pair $(x_\beta,q_\beta)\in\X\times\mathbb Q$, such that $E_{\beta-1}(x_\beta)<q_\beta<E_{\beta}(x_\beta)$. This defines an injection $B\to\X\times\mathbb Q$, whose existence implies that $B$ is countable, because $\X\times\mathbb Q$ is countable. As every countable subset of $\Omega_0$ has a least upper bound in $\Omega_0$ (Theorem 1.14 in \citealp{aliprantis06infinite}), we have $\beta^\star = \sup B\in\Omega_0$. By construction, $E_{\beta^\star}\succeq E_1$ and $E_{\beta^\star}$ is maximal, since it must be that $E_{\beta^\star+1} = E_{\beta^\star}$, as otherwise $\beta^\star+1$ would be in $B$.
\end{proof}

\begin{manual}{\Cref{prop:nonex}}
    Let $\X=[0,1]$. Consider the hypothesis
    $$\Hh = \big\{P\in\PR_{[0,1]}\,:\,P(\{0\})\geq 1/2\big\}\cup\{U_{[0,1]}\}\,,$$
    with $U_{[0,1]}$ the uniform distribution on $[0,1]$. $\Hh$ does not admit an optimal e-class.
\end{manual}
\begin{proof}
    First, let us notice that all e-variables must be everywhere upper bounded by $2$. Indeed, for any $x\in[0,1]$ we have that $P_x = \frac{1}{2}\delta_0 + \frac{1}{2}\delta_x$ is in $\Hh$. In particular, for any e-variable $E$ we have $\langle P_x, E\rangle = \frac12(E(x) + E(0)) \geq \frac12 E(x)$, and so $E(x)\leq 2$. 
    
    Let $\Ee_1$ denote the set of all e-variables such that $E(1)=2$. Clearly, this is non-empty, since the function that is $0$ everywhere and $2$ on $1$ belongs to $\Ee_1$. We now show that no variable in $\Ee_1$ can be maximal. As a consequence, since all e-variables in $\Ee_1$ can only be majorised by e-variables in $\Ee_1$, we conclude that there is no maximal e-variable that majorises the e-variables in $\Ee_1$, and so the set of all maximal e-variables is not a majorising e-class. By \Cref{lemma:maxopt} we conclude that the optimal e-class does not exist. 

    Now, let $E\in\Ee_1$. First note that $E(0) = 0$, as with $P_1 = \frac{1}{2}\delta_0 + \frac12\delta 1$ we have that $1\geq \frac{1}{2}(E(0) + E(1)) = 1 + \frac{1}{2}E(0)$. Moreover, since $\langle U_{[0,1]},E\rangle\leq 1$, $E$ cannot be identically equal to $2$ and there must be at least one $x'\in(0,1]$ such that $E(x')<2$. Define $E'$ as 
    $$E'(x) = \begin{cases}E(x)&\text{if $x\neq x'$;}\\2&\text{if $x=x'$.}\end{cases}$$
    Clearly, $E'\succ E$. So, it is enough to show that $E'$ is an e-variable to establish that $E$ cannot be maximal. 
    
    Fix any $P$ such that $P(\{0\})\geq 1/2$. We can write $P = P(\{0\})\delta_0 + (1-P(\{0\}))P'$, where $P'\in\PR_{[0,1]}$ is such that $P'(\{0\}) = 0$. We have 
    $$\langle P, E'\rangle = P(\{0\})E'(0) + (1-P(\{0\}))\langle P', E'\rangle\,.$$ Since $x'>0$, $E'(x)=E(x)=0$, so $\langle P, E'\rangle \leq \frac{1}{2}\langle P', E'\rangle\leq 1$, since $E'\preceq 2$. We are only left with checking that $\langle U_{[0,1]}, E'\rangle \leq 1$. But since $E$ and $E'$ differs only on $x'$, it is clear that $\langle U_{[0,1]}, E'\rangle = \langle U_{[0,1]}, E\rangle$, and so we conclude.
\end{proof}

\begin{manual}{\Cref{lemma:prored}}
    Let $\Hh$ be a properly constrained hypothesis on $\X$. Then, for every $x\in\X$ there exists a $P\in\Hh$ whose support has finite cardinality and such that $x\in\supp P$.
\end{manual}

\begin{proof}
Let $\Phi:\X\to\R^m$ be a proper constraint and denote as $\Y$ its image. Now, fix $x_0\in\X$ and let $y_0=\Phi(x_0)$. By \Cref{cor:Yfinsupp}, since $0\in\relint(\conv\Y)$, we can find $P'\in\PR_\Y$, with zero mean and such that $P'(\{y_0\})>0$, and whose support has finitely many (say $r+1$) elements. So, for some $r\geq 0$ we can write $\supp P' = \{y_0, y_1, \dots, y_r\}$. We can choose $r$ points $\{x_1,\dots,x_r\}\subseteq \X$ such that $\Phi(x_i) = y_i$ for any $i=1\dots r$. Then, we define $P = \sum_{i=0}^r P'(\{y_i\})\delta_{x_i}$. Clearly, $\supp P$ has finitely many elements (so, $\langle P, \|\Phi\|_1\rangle<+\infty$) and $P(\{x_0\}) = P'(\{y_0\})>0$. Moreover, $\langle P, \Phi\rangle = \langle P', Y\rangle =0$ by construction, so $P\in\Hh$.
\end{proof}

\begin{manual}{\Cref{lemma:finicovpro}}
    Let $\Hh$ be a finitely constrained hypothesis on $\X$. $\Hh$ is properly constrained if, and only if, $\sup_{P\in\Hh}P(\{x\})>0$ for all $x\in\X$. Moreover, if $\Hh$ is properly constrained, every constraint $\Phi$ for $\Hh$ is a proper constraint and there is $A:\X\to[1, +\infty)$ such that $E\preceq A$ for any $E\in\Ee_\Hh$.
\end{manual}
\begin{proof}
    Let $\Phi$ be a constraint for $\Hh$. Note that if $\Phi$ is identically null, then $\Hh = \PR_\X$ and $\Phi$ is a proper constraint, as $\relint\{0\} = \{0\}$. So, we only need to consider the case where $\Phi$ is not identically null.
    
    First, if $\Hh$ is properly constrained, by \Cref{lemma:prored}, $\sup_{P\in\Hh} P(\{x\})>0$, for all $x\in\X$. 
    
    Now, assume that $\sup_{P\in\Hh} P(\{x\})>0$ for all $x\in\X$, and let $\Hh$ be a finitely constrained hypothesis. Let $\Phi$ be a constraint. We show that $\Phi$ is a proper constraint by contradiction. Denote as $\Cc$ the convex hull of $\Phi(\X)$, and let $\partial\Cc$ be its relative boundary. Assume that $\Phi$ is not a proper constraint. Then, $0\in\partial\Cc$. Since $\Hh$ is non-empty and $\Phi$ is not identically null, there are at least two distinct points in $\Phi(\X)$. Hence, $\Cc$ has non-empty relative interior and there is $x_0\in\X$ such that $\Phi(x_0)\in\relint\Cc$. By assumption, there is $P\in\Hh$ such that $P(\{x_0\})>0$, so $x_0\in\supp P$. But by \Cref{lemma:convP}, $\supp P\subseteq\Phi^{-1}(\partial\Cc)$ for each $P\in\Hh$. This is a contradiction since $x_0\in\Phi^{-1}(\relint\Cc)$, which is an open set\footnote{Indeed, we can see $\Phi$ as a continuous function from $\X$ to $\aff\Phi(\X)$, and $\relint\Cc$ is an open set of $\aff\Phi(\X)$, so its counter-image is open.} that has an empty intersection with $\Phi^{-1}(\partial\Cc)$.

    Note that what shown so far also implies that any constraint of a properly constrained hypothesis is a proper constraint.

    Finally, let $\Hh$ be properly constrained and $A:x\mapsto 1/\sup_{P\in\Hh}P(\{x\})$, which defines a function $\X\to[1,+\infty)$. Then, for any e-variable $E\in\Hh$ and any $x\in\X$, we have $$E(x)\leq A(x)E(x)\sup_{P\in\Hh}P(\{x\})\leq A(x)\sup_{P\in\Hh}\langle P, E\rangle \leq A(x)\,,$$ and we conclude.
\end{proof}

\begin{manual}{\Cref{lemma:prominiff}}
    Let $\Hh$ be a properly constrained hypothesis and $\Phi$ a proper constraint. $\Phi$ is minimal if, and only if, all its components are linearly independent scalar functions.
\end{manual}
\begin{proof}
    Let $\dim(\Span\Phi)=m'$.
    By \Cref{lemma:dimaff}, $m'$ is also the dimension of the affine hull of $\Phi(\X)$. Hence, $\aff\Phi(\X)=\R^m$ if, and only if, $m'=m$. If this is the case, then the interior and relative interior of $\conv\Phi(\X)$ coincide, and so $0\in\relint(\conv\Phi(\X))$ implies that $0\in\inter(\conv\Phi(\X))$. On the other hand, if $m'<m$ the interior of $\conv\Phi(\X)$ is empty, and $\Phi$ cannot be minimal.
\end{proof}

\begin{manual}{\Cref{lemma:prominc}}
    Every properly constrained hypothesis admits a minimal constraint.
\end{manual}
\begin{proof}
    Let $\Hh$ be a properly constrained hypothesis on $\X$ and $\Phi:\X\to\R^m$ be a proper constraint. Denote as $U$ the affine hull $\aff\Phi(\X)$, and let $\Pi_U$ be the orthogonal projection from $\R^m$ to $U$. Let $\Phi_U = \Pi_U\circ\Phi$ and $\Hh_U = \{P\in\PR_{\Phi_U}\,:\,\langle P, \Phi_U\rangle =0\}$. Since $\Phi(\X)\subseteq U$, we have that $\Phi_U(x) = \Phi(x)$ for all $x\in\X$. It follows immediately that $\Hh=\Hh_U$. Using that $U\cong\R^{m'}$ (for some $m'\leq m$), we can define a linear isomorphism $I:U\to\R^{m'}$. Then, $\langle P, \Phi_U\rangle = 0$ if, and only if, $\langle P, I\circ\Phi_U\rangle =0$. It follows that $I\circ\Phi_U$ is a proper constraint. By construction, $0\in\inter(\conv(I\circ\Phi_U(\X)))$, so $I\circ\Phi_U$ is a minimal constraint.
\end{proof}

\begin{manual}{\Cref{lemma:profull}}
    Let $\Hh$ be a finitely constrained hypothesis on $\X$. $\Hh$ is properly constrained if, and only if, there is $P\in\Hh$ such that $\supp P=\X$. 
\end{manual}
\begin{proof}
    Assume that $\Hh$ is properly constrained. Fix a sequence $(x_i)_{i\geq 1}\subseteq\X$, dense in $\X$. For each $i\geq 1$, by \Cref{lemma:finicovpro} there is a $P_i\in\Hh$ such that $P_i(\{x\})>0$. Let $\Phi$ be a constraint. Since $P_i\in\Hh$ we have $P\in\PR_\Phi$, and so $\zeta_i = \langle P_i,\|\Phi\|_1\rangle$ is finite. Let $\alpha_i = 2^{-i}/(1+\zeta_i)$. Then, we have $\xi=\sum_{i=1}^\infty\alpha_i\in(0,1]$. Let $P = \sum_{i=1}^\infty \frac{\alpha_i}{\xi}P_i$. By construction, $P\in\PR_\X$. Moreover, $\langle P, \|\Phi\|_1\rangle = \xi\sum_{i=1}^\infty \alpha_iz_i\leq \xi\sum_{i}^\infty 2^{-i}=\xi$, and $\langle P, \Phi\rangle=0$. It follows that $P\in\Hh$. As, $\supp\Phi\supseteq(x_i)_{i\geq 1}$, which is dense, we conclude that $\supp P = \X$.

    Conversely, let $\Hh$ be finitely constrained and not properly constrained. For every constraint $\Phi$, it must be that $0\in\Cc=\conv\Phi(\X)$ (see \Cref{lemma:convP}). Since there is no proper constraint, $0$ must be in the relative boundary $\partial\Cc$ of $\Cc$. Again by \Cref{lemma:convP}, any $P\in\Hh$ is supported on $\Phi^{-1}(\partial\Cc)$. As the relative interior of a non-empty set is non-empty, $\supp P\subseteq \Phi^{-1}(\partial\Cc)\subsetneq\X$.
\end{proof}

\begin{manual}{\Cref{lemma:Hsfic}}
    Let $\Hh$ be a finitely constrained hypothesis on $\X$ and $\Phi$ a constraint. A closed set $S\subseteq\X$ is compatible if, and only if, $0\in\conv\Phi(S)$. In such case, $\Hh_S$ is a finitely constrained hypothesis on $S$ and $\Phi\big|_S$ is a constraint for it. In particular, $\Hh_S = \{P\in\PR_{\Phi|S}\,:\,\langle P,\Phi\big|_S\rangle=0\}$.
\end{manual}
\begin{proof}
    If $0\in\conv\Phi(S)$, \Cref{lemma:car} implies that there is a finite set $\{x_1,\dots,x_r\}\subseteq S$ such that $\sum_{i=1}^r\alpha_i\Phi(x_i) = 0$, with the $\alpha_i$ non-negative coefficients that sum to $1$. In particular, $P=\sum_{i=1}^r\alpha_i\delta_{x_i}$ is in $\Hh_S$, and so $S$ is compatible. On the other hand, if $S$ is compatible, then $0\in\conv\Phi(S)$ by \Cref{lemma:convP}.

    Now, let $\Phi$ be a constraint for $\Hh$. Then $\Phi\big|_S:S\to\R^m$ is continuous, as it is the restriction of a continuous function.  Let $P\in\Hh_S$. Then, $P\in\PR_\Phi\cap\PR_S = \PR_{\Phi|S}$, and since $\Hh_S\subseteq\Hh$, we have that $\langle P, \Phi\big|_S\rangle = \langle P, \Phi\rangle = 0$. So, $\Hh_S\subseteq \{P\in\PR_{\Phi|S}\;:\;\langle P, \Phi\big|_S\rangle=0\}$.
    Now, if $P\in\PR_{\Phi|S}$ and $\langle P, \Phi\big|_S\rangle =0$, then $\langle P, \Phi\rangle =0$. So $P\in\Hh$ and, since $P\in\PR_S$, we also have $P\in\Hh_S$. 
\end{proof}

\begin{manual}{\Cref{lemma:comsupp}}
    Let $\Hh$ be a properly constrained hypothesis on $\X$ and $S\subseteq\X$ a compatible set. $S$ is a matching set if, and only if, $\Hh_S$ is a properly constrained hypothesis on $S$.
\end{manual}
\begin{proof}
    This is a direct consequence of \Cref{lemma:profull}.
\end{proof}
\begin{manual}{\Cref{lemma:finprocont}}
    Let $\Hh$ be a properly constrained hypothesis on $\X$. A finite union of matching sets is a matching set. In particular, every finite subset of $\X$ is contained in a finite matching set.
\end{manual}
\begin{proof}
    Let $\{S_1,\dots, S_N\}$ be a finite family of matching sets. For each $S_i$ there is $P_i\in\Hh$ such that $\supp P_i=S_i$. Let $P = N^{-1}\sum_{i=1}^N P_i$. Then, $P\in\Hh$ by convexity, so $S'=\supp P$ is a matching set. Moreover, $S'$ is the closure of $\bigcup_{i=1}^N S_i$, but since every $S_i$ is closed we have $S'=\bigcup_{i=1}^m S_i$. Now, let $S\subseteq\X$ be any finite set. By \Cref{lemma:prored}, for every $x\in S$ there is a $P_x\in\Hh$ with finite support $S_x$ that contains $x$. Clearly, $\bigcup_{x\in S}S_x$ contains $S$, and it is a matching set from what we have just shown, as it is the union of finitely many matching sets.
\end{proof}

\begin{manual}{\Cref{lemma:mincomp}}  
    Let $\Hh$ be a properly constrained hypothesis on $\X$ and $\Phi$ a constraint. Then $\Phi$ is minimal, if, and only if, $\Lambda_\Phi$ is compact. More precisely, $\Phi$ is minimal if, and only if, there is a finite matching set $S^\star$ such that $\Lambda_{\Phi|S^\star}$ is compact. If such $S^\star$ exists, $\Lambda_{\Phi|S}$ is compact for any $S\subseteq\X$ that contains $S^\star$, and any closed set $S\subseteq \X$ that contains $S^\star$ is a matching set.
\end{manual}
\begin{proof}
    First, notice that if $\Phi:\X\to\R^m$ is not minimal, then by \Cref{lemma:prominiff} its components are not linearly independent. In particular, there is $\lambda\in\R^m$ such that $\lambda\cdot\Phi$ is identicalluy null but $\lambda\neq0$. Since, for any $\alpha\in\R$, we have that $\alpha\lambda\in\Lambda_\Phi$, we conclude that $\Lambda_\Phi$ is unbounded. In particular, for any $S^\star\subseteq\X$ we also have that $\Lambda_{\Phi|S^\star}\supseteq\Lambda_\Phi$ is unbounded.

    Now, assume that $\Phi:\X\to\R^m$ is minimal. Again by \Cref{lemma:prominiff}, all its components $\Phi_i:\X\to\R$ are linearly independent. In particular, by \Cref{lemma:ludovic} we can find a set $s$ with finitely many elements, such that $\{\Phi_1\big|_{s},\dots,\Phi_m\big|_{s}\}$ is a family of $m$ linearly independent finite-dimensional vectors. By \Cref{lemma:finprocont}, there is a matching set $S^\star\subseteq \X$ that contains $s$ and has finitely many elements. Clearly, we still have that $\{\Phi_1\big|_{S^\star},\dots,\Phi_m\big|_{S^\star}\}$ is a linearly independent family. As $\Hh_{S^\star}$ is a properly constrained hypothesis by \Cref{lemma:comsupp}, and $\Phi\big|_{S^\star}$ a constraint for it (\Cref{lemma:Hsfic}), $1-\lambda\cdot\Phi\big|_{S^\star}$ is an e-variable for $\Hh_{S^\star}$ for every $\lambda\in\Lambda_{\Phi|{S^\star}}$, as it is non-negative and has expectation $1$ under every $P\in\Hh_{S^\star}$. In particular, since ${S^\star}$ is finite and $\Hh_{S^\star}$ is properly constrained, by \Cref{lemma:finicovpro} there is a finite constant $a\geq 1$ such that, for any $x\in S^\star$ and any $\lambda\in\Lambda_{\Phi|{S^\star}}$, we have $1-a\leq \lambda\cdot\Phi(x)\leq 1$. By \Cref{lemma:boundedL}, $\Lambda_{\Phi|{S^\star}}$ is bounded, and so compact by \Cref{lemma:LFclo}. 
    
    Now, let $S\subseteq\X$ be any set that contains $S^\star$. Then, $\Lambda_{\Phi|S}$ is compact, as it is a closed (\Cref{lemma:LFclo}) subset of $\Lambda_{\Phi|{S^\star}}$. In particular, $\Lambda_\Phi$ is compact. Moreover, from what we have already shown, $\Phi\big|_{S^\star}$ is a minimal constraint for $\Hh_{S^\star}$, since $\Lambda_{\Phi|S^\star}$ is compact. We thus have that $0\in\inter(\conv\Phi(S^\star))\subseteq\inter(\conv\Phi(S))$. In particular, by \Cref{lemma:Hsfic}, $S$ is compatible, and $\Hh_S$ is a finitely constrained hypothesis on $S$, with $\Phi\big|_S$ a constraint. Since $\Phi\big|_S$ contains $0$ in the interior of the convex hull of its image, it is a minimal constraint, and so $\Hh_S$ is properly constrained on $S$. By \Cref{lemma:comsupp} we conclude that $S$ is a matching set.
\end{proof}

\begin{manual}{\Cref{lemma:X0}}
    Let $\Hh$ be a finitely constrained hypothesis on $\X$. Let $\X_0=\bigcup_{P\in\Hh}\supp P$. Then, $\X_0$ is closed and $\Hh$ is a properly constrained hypothesis on $\X_0$.
\end{manual}
\begin{proof}
    Let $\bar\X_0$ be the closure of $\X_0$. Since for every $P\in\Hh$ we have that $\supp P\subseteq\bar\X_0$, and $\bar\X_0$ is closed, $\Hh\subseteq\PR_{\bar\X_0}$ and so it is a hypothesis on $\bar\X_0$. In particular, $\bar\X_0$ is a compatible set, and $\Hh_{\bar\X_0} = \Hh\cap\PR_{\bar\X_0} = \Hh$. Hence, $\Hh$ is a finitely constrained hypothesis on $\bar\X_0$ by \Cref{lemma:Hsfic}. By \Cref{lemma:profull}, to show that $\Hh$ is a properly constrained hypothesis on $\bar\X_0$ it is enough to show that there is $P^\star\in\Hh$ such that $\supp P^\star = \bar\X_0$. As $\supp P^\star\subseteq\X_0\subseteq\bar\X_0$, this will also yield $\X_0=\bar\X_0$.
    
    Fix a sequence $(x_i)_{i\geq 1}\subseteq\X_0$, dense in $\bar\X_0$. By construction, for each $i\geq 1$, there is $P_i\in\Hh$ such that $x_i\in\supp P_i$. Let $\Phi:\X\to\R^m$ be a constraint for $\Hh$. Since $P_i\in\Hh$, we have $P_i\in\PR_{\Phi}$, and so $\zeta_i = \langle P_i, \|\Phi\|_1\rangle$ is finite. Let $\alpha_i = 2^{-i}/(1+\zeta_i)$. Then, we have $\xi=\sum_{i=1}^\infty\alpha_i\in(0, 1]$. Let $P^\star = \sum_{i=1}^\infty\frac{\alpha_i}{\xi}P_i$. Clearly, $P^\star\in\PR_\X$. Moreover, $\langle P^\star, \|\Phi\|_1\rangle = \xi\sum_{i=1}^\infty\alpha_i\zeta_i\leq\xi\sum_{i=1}^\infty 2^{-i}=\xi$. It follows that $P^\star\in\PR_\Phi$. As $\langle P^\star,\Phi\rangle =0$, $P^\star\in\Hh$. Moreover, $\supp P^\star\supseteq \bigcup_{i=1}^\infty\supp P_i\supseteq (x_i)_{i\geq 1}$. Since $(x_i)_{i\geq 1}$ is dense in $\bar\X_0$ and $P^\star\in\PR_{\bar\X_0}$ by construction, we deduce that $\supp P^\star = \bar\X_0$. 
\end{proof}

\section{Technical results}\label{app:tech}

\subsection{Some linear algebraic results}\label{app:linalg}

\begin{lemma}\label{lemma:ludovic}
    Let $F=\{f_1,\dots,f_m\}$ be a family of $m\geq1$ functions $\X\to\R$. $F$ is a linearly independent family if, and only if, there is a set $S=\{x_1,\dots,x_m\}\subseteq\X$ such that $\{f_1\big|_S,\dots,f_m\big|_S\}$ is a linearly independent family of functions $S\to\R$.
\end{lemma}
\begin{proof}
    If such $S$ exists, then $F$ is an independent family. Indeed, if for some real coefficients $\sum_{i=1}^m\alpha_if_i = 0$, then $\sum_{i=1}^n\alpha_if_i\big|_S = 0$, which implies that all the coefficients must be null.
    
    We will prove the reverse implication by induction. Obviously, if $m=1$ the statement is true, as $f_1$ must be non-zero somewhere for $F$ to be an independent family, and it is enough to set $S=\{x_1\}$, with $f(x_1)\neq 0$. Now, fix $m\geq2$ and assume that the statement holds for $m-1$. Then, we can find a set $S'=\{x_1,\dots,x_{m-1}\}$ such that the restriction of $\{f_1,\dots, f_{m-1}\}$ to $S'$ is a family that spans the $m-1$ dimensional space of the functions from $S'$ to $\R$. In particular, there is a unique vector of coefficients $(\lambda_1,\dots,\lambda_{m-1})$ such that $f_m\big|_{S'} = \sum_{i=1}^{m-1}\lambda_if_i\big|_{S'}$. Since $f_m$ is independent of $\{f_1,\dots,f_{m-1}\}$, it cannot be that $f_m=\sum_{i=1}^{m-1}f_i$ on the whole $\X$, and hence there must exist $x_m\in\X\setminus S'$ such that $f_m(x_m)\neq\sum_{i=1}^{m-1}\lambda_if_i(x_m)$. Let $S=\{x_1,\dots,x_m\}$. Now, let $(\alpha_1,\dots,\alpha_m)\in\R^m$ be such that $\sum_{i=1}^m \alpha_if_i\big|_{S}=0$. In particular, we have that $\sum_{i=1}^{m-1}(\alpha_i+\alpha_m\lambda_i)f_i\big|_{S'}=0$, which implies that for every $i$ from $1$ to $m-1$ it must be that $\alpha_i = -\alpha_m\lambda_i$. Hence, we get that $$0=\sum_{i=1}^m\alpha_i f_i(x_m)=\alpha_mf_m(x_m)-\alpha_m\sum_{i=1}^{m-1}\lambda_if_i(x_m) = \alpha_m\left(f_m(x_m)-\sum_{i=1}^{m-1}\lambda_if_i(x_m)\right)\,.$$ But, we have chosen $x_m$ such that $f_m(x_m)-\sum_{i=1}^{m-1}\lambda_if_i(x_m)\neq0$, and so $\alpha_m=0$. It follows that all the coefficients $\alpha_i$ are null, and so $\{f_1\big|_S,\dots,f_m\big|_S\}$ is an independent family. 
\end{proof}

\begin{lemma}\label{lemma:boundedL}
    Let $\{v_1,\dots,v_r\}\subseteq\R^m$ be a family of vectors that spans $\R^m$. Let $\{(A_i, B_i)\}_{i=1}^r$ be a family of non-empty bounded intervals. Then, the set 
    $\Lambda = \{\lambda\in\R^m\;:\;\lambda\cdot v_i\in[A_i, B_i]\,,\;\forall i=1\dots r\}$ is bounded.
\end{lemma}
\begin{proof}
    Since $\{v_1,\dots, v_r\}$ spans $\R^m$, each element $e_i$ of the standard basis of $\R^m$ can be written as a linear combination $e_i = \sum_{j=1}^r\alpha_{ij}v_j$. Then, if $\lambda\cdot v_j$ is in a bounded interval for each $j$, it follows that $\lambda\cdot e_i$ is also in a bounded interval. Hence, all the components of $\lambda$ must be bounded if $\lambda\in\Lambda$.
\end{proof}

\subsection{Some results from convex analysis}\label{app:convex}
First, we state some standard definitions for elementary notions in convex analysis. We refer to \cite{hiriart2004fundamentals} for a thorough introduction to the subject. 
Given a set $\Y\subseteq\R^n$, its \emph{convex hull} $\Cc=\conv\Y$ is defined as the smallest convex set containing $\Y$, or equivalently, as the intersection of all the convex sets that contain $\Y$. The \emph{affine hull} $\A=\aff\Y$ of $\Y$ is the smallest affine space in $\R^n$ that contains $\Y$, while the \emph{linear hull} (or \emph{span}) $\Ll=\lin\Y$ of $\Y$ is the smallest vector subspace of $\R^n$ containing $\Y$. Clearly, $\Y\subseteq\Cc\subseteq\A\subseteq\Ll$. 
Let $C\subseteq\R^n$ be a convex set and $\A$ its affine hull. The \emph{relative interior} of $C$ (denoted as $\relint C$) is the interior of $C$ with respect to the induced topology on $\A$. More explicitly, $x\in C$ is in $\relint C$ if, and only if, there is an open set $U$ in $\R^n$ containing $x$ and such that $U\cap\A\subseteq C$. The \emph{relative boundary} of $C$ is the set $C\setminus\relint C$.

\begin{lemma}[Carathéodory Theorem]\label{lemma:car}
    Let $\Y$ be a set in $\R^n$ and $\Cc$ its convex hull. Denote as $m$ the affine dimension of $\Y$ (namely, the dimension of its affine hull). Let $y\in \Cc$. Then, there is a set $S\subseteq\Y$ with $m+1$ elements, such that $y\in\conv S$. Equivalently, there is a Borel probability measure $P\in\PR_\Y$ such that $\langle P, Y\rangle =y$ and whose support has at most $m+1$ elements.
\end{lemma}
\begin{proof}
    This is a classical elementary result in convex analysis, attributed to Carathéodory.
\end{proof}

\begin{corollary}\label{cor:Yfinsupp}
    Let $\Y$ be a set in $\R^n$ and $\Cc$ its convex hull. Fix any $y_0\in\relint\Cc$. For any $y\in\Y$ there is $P\in\PR_\Y$ whose support has finite cardinality and contains $y$, and such that $\langle P, Y\rangle =y_0$. 
\end{corollary}
\begin{proof}
    If $y=y_0$, the statement is trivial, as we can choose $P=\delta_y$. So, consider the case $y\neq y_0$. As $y_0\in\relint\Cc$, we can find $\ee>0$ such that the closed ball $B_\ee(y_0)$, centred in $y_0$, satisfies $B_\ee(y_0)\cap\aff\Cc\subseteq\Cc$. Let $\alpha = \ee/\|y-y_0\|$, with $\|\cdot\|$ the standard Euclidean norm. Then, we have that $z = y_0 - \alpha(y-y_0)$ is in $\Cc$, and we can apply \Cref{lemma:car} to find a measure $P'\in\PR_\Y$ with mean $z$ and finite support. Now, let $P = 
    \frac{\alpha}{1+\alpha}\delta_y + \frac{1}{1+\alpha} P'$. Then, $\langle P, Y\rangle = y_0$, and by construction $P(\{y\}) \geq \frac{\alpha}{1+\alpha}>0$.
\end{proof}

\begin{lemma}\label{lemma:dimaff}
    Let $\Phi:\X\to\R^m$ be any function. If $0\in\conv\Phi(\X)$, $\dim(\aff\Phi(\X)) = \dim(\Span\Phi)$.
\end{lemma}
\begin{proof}
    First, note that since $0\in\conv\Phi(\X)$, we have that $0\in\aff\Phi(\X)$, which implies that $\lin\Phi(\X)=\aff\Phi(\X)$.
    So, it is enough to prove that $\dim(\lin\Phi(\X)) = \dim(\Span\Phi)$. 

    Let $m'$ be the dimension of $\Span\Phi$. We can find $m'$ components of $\Phi$ that are linearly independent. Without loss of generality, we assume that these are the first $m'$ components, and let $\Phi' = (\Phi_1,\dots,\Phi_{m'})$, which is a function $\X\to\R^{m'}$ with $\Span\Phi'$ of dimension $m'$. By \Cref{lemma:ludovic}, there is a family $S = \{x_1,\dots, x_{m'}\}$ of $m'$ distinct points in $\X$ such that the $m'\times m'$ Gram matrix with components $\Phi_i(x_j)$ (for $i$ and $j$ ranging from $1$ to $m'$) has full rank. In particular, $\lin \Phi'(S)=\R^{m'}$. Since $\R^{m'}\supseteq\lin \Phi'(\X)\supseteq\lin\Phi'(S)$, we conclude that $\dim(\lin\Phi'(\X))=m'$. So, $\dim(\lin\Phi(\X))\geq m'$. \Cref{lemma:ludovic} also implies that no subset of $\X$ with strictly more than $m'$ elements yields a fully ranked Gram matrix. So, there are at most $m'$ independent vectors in $\lin\Phi(\X)$, and so $\dim(\lin\Phi(\X))=m'$.
\end{proof}

\begin{lemma}\label{lemma:convP}
    Let $\X$ be a closed set in $\R^n$ and $\Phi:\X\to\R^m$ a Borel function. Let $\Cc=\conv\Phi(\X)$, and $\partial\Cc$ its relative boundary. If there is $P\in\PR_\Phi$ such that $\langle P, \Phi\rangle = 0$, then $0\in\Cc$. Moreover, if $0\in\partial\Cc$, then any $P\in\PR_\Phi$ such that $\langle P,\Phi\rangle =0$ satisfies $\supp P\subseteq \Phi^{-1}\big(\bigcap_{\pi\in\Pi_0}\pi\big)$, where $\Pi_0$ is the set of all supporting hyperplanes of $\Cc$ containing $0$, and so in particular $\supp P\subseteq \Phi^{-1}(\partial\Cc)$.
\end{lemma}
\begin{proof}
    We will show the first statement by induction on $m$. If $m=0$, $\Cc = \{0\}=\R^0$, and the result is trivial. Now, for $m\geq 1$, assume that the result holds for $m-1$. We prove by contradiction that it must hold for $m$. Assume that $0\notin \Cc$. By the hyperplane separating theorem for convex sets, there is a vector $v\in\R^m$ such that $v\cdot y\geq 0$, for every $y\in \Cc$. Since $\Phi(\X)\subseteq \Cc$, any $P$ such that $\langle P, \Phi\rangle=0$ must be supported on the set $\X_0=\{x\in\X\,:\,\Phi(x)\cdot v=0\}$. Since $\X_0\cong\R^{m-1}$, we conclude that $0\in\conv\Phi(\X_0)$ by the inductive hypothesis, which is a contradiction since $\conv\Phi(\X_0)\subseteq\Cc$.

    For the second claim, assume that $0\in\partial\Cc$ and fix any $P\in\PR_\Phi$ such that $\langle P, \Phi\rangle = 0$. $\Pi_0$ is non-empty as $0\in\partial\Cc$, and for any supporting hyperplane $\pi\in\Pi_0$ there is a linear map $\phi$ such that $\pi = \{\phi=0\}$ and $\Phi(\X)\subseteq\Cc\subseteq\{\phi\geq 0\}$. Hence, $\supp P\subseteq\Phi^{-1}(\pi) = \Phi^{-1}(\pi\cap\Phi(\X))\subseteq\Phi^{-1}(\partial\Cc)$, since $\pi\cap\Phi(\X)\subseteq\partial\Cc$. As this is true for every $\pi\in\Pi_0$ we also have that $\supp P\subseteq\Phi^{-1}\big(\bigcap_{\pi\in\Pi_0}\pi\big)$.
\end{proof}

\begin{lemma}\label{lemma:relintsupp}
    Let $\X$ be a set with finite cardinality $d\geq 1$. Then, the relative interior of $\PR_\X$ (which can be seen as a subset of $\R^d$) is the set $\{P\in\PR_\X\,:\,\supp P=\X\}$.
\end{lemma}
\begin{proof}
    Let $\R^d_{\geq0}$ denote the subset of $\R^d$ of vectors whose components are all non-negative. Its relative interior is $\R^d_{>0}$, the set of vectors with only strictly positive components. Let $\one{}\in\R^d$ denote the vector with all components equal to one and $V=\{v\in\R^d\,:\,v\cdot\one{}=1\}$, whose relative interior is $V$ itself. Then $\PR_\X = \R^d_{\geq0}\cap V$, which is the non-empty intersection of two convex sets. Since $\relint V\cap \relint \R^d_{\geq0}\neq\varnothing$, from Proposition 2.1.10 in Chapter A of \cite{hiriart2004fundamentals}, we get $$\relint\PR_\X = \relint V\cap \relint \R^d_{\geq0} = V\cap\R^d_{>0} = \PR_\X\cap\R^d_{>0}\,,$$ which is precisely the set of fully supported probability measures on $\X$.
\end{proof}

\begin{lemma}\label{lemma:span}
    Let $V\subseteq\R^d$ be a linear subspace. Denote as $\Delta$ the simplex in $\R^d$ and as $\Delta^\circ$ its relative interior. If $V\cap\Delta^\circ\neq\varnothing$, we have that $\lin(V\cap\Delta) = V$. 
\end{lemma}
\begin{proof}
    Everything is trivial if $d=0$, so assume $d\geq 1$. Let $v\in V\cap\Delta^\circ\neq\varnothing$. As $v\in\Delta^\circ$, by \Cref{lemma:relintsupp} all its component are strictly positive, so there is an open ball $B\subset\R^d_{\geq 0}$ centred at $v$. It follows that 
    $$V=\lin V \supseteq\lin(V\cap\Delta)=\lin(V\cap\R^d_{\geq 0})\supseteq\lin(V\cap B)\,.$$
    However, since $v\in V$, $\lin(V\cap B) = \lin(V\cap B_0)=V$, where $B_0 = B-v$ is the open ball with same radius as $B$ and centred in $0$. So, we conclude.
\end{proof}

\section{Proof of \texorpdfstring{\Cref{prop:loosely}}{Proposition 9.3}}\label{app:loose}
Here we provide a proof of \Cref{prop:loosely}. First, we discuss the case of $\X$ with finite cardinality. Much of the proof follows that of \Cref{lemma:finopt}, but for the final step a different path needs to be taken as it is not anymore true that $\langle P, E\rangle =1$ for any $P\in\Hh$ if $E$ is maximal. 

We start by an easy lemma.
\begin{lemma}\label{lemma:logE}
    Let $\Hh$ be a properly loosely constrained hypothesis on $\X$ and $\X$ have finite cardinality. Then, for any $Q\in\PR_\X$, the set $\argmax_{E\in\Ee_\Hh}\langle Q, \log E\rangle$ is non-empty. Moreover, all its elements coincide on $\supp Q $.
\end{lemma}
\begin{proof}
    Fix $Q\in\PR_\X$. Since $\Ee_\Hh$ is compact and $\Phi:E\mapsto\langle Q,\log E\rangle=\sum_{x\in\supp Q }Q(\{x\})\log E(x)$ is, upper semi-continuous, concave, and valued on $[-\infty, +\infty)$, we have that $\argmax_{E\in\Ee_\Hh}\Phi(E)$ is non-empty. Moreover, since the constant function $1$ is an e-variable, the maximum of $\Phi$ is always finite. Now, denoting as $\Pi_Q$ the projection of $E\in\Ee_\Hh$ on $\supp Q $ (namely $\Pi_Q(E)(x) = E(x)$ if $x\in\supp Q $ and $0$ otherwise) $\Phi$ is strictly convex on $\Pi_Q(\Ee_\Hh)$, which is a convex set. Hence, $\argmax_{E\in\Pi_Q(\Ee_\Hh)}\Phi(E)$ is unique, and so all the elements in $\argmax_{E\in\Ee_\Hh}\Phi(E)$ must coincide on $\supp Q $.
\end{proof}
Now, we recall a result, which follows directly from \cite{grunwald2024safe} and \cite{larsson2024numeraire}. 
\begin{lemma}\label{lemma:gru}
    Let $\Hh$ be a properly loosely constrained hypothesis on $\X$, and assume that $\X$ has finite cardinality. Then, there exists a unique minimiser $P_Q\in\Hh$ of $P\mapsto\KL(Q|P)$, whose support contains $\supp Q $. Let $E_Q:\X\to\R_{\geq 0}$ be defined as $$E_Q(x)=\begin{cases}\tfrac{Q(\{x\})}{P_Q(\{x\})}&\text{if $x\in\supp Q$;}\\0&\text{otherwise.}\end{cases}$$
    $E_Q$ is an e-variable, and $E_Q\in\argmax_{E\in\Ee_\Hh}\langle Q, \log E\rangle$. Moreover, if for an e-variable $E$ there is $P\in\Hh$ such that $E(x)=Q(\{x\})/P(\{x\})$, for all $x\in\supp Q $, then $E$ and $E_Q$ must coincide on $\supp Q$.
\end{lemma}
We can now show that if $\X$ has finite cardinality and $E$ is a maximal e-variable for a properly loosely constrained $\Hh$, then $E$ must belong to the dual e-class.
\begin{lemma}\label{lemma:loosefin}
Let $\X$ have finite cardinality. Let $\Hh$ be a properly loosely constrained hypothesis on $\X$, with $\Phi':\X\to\R^{m'}$ and $\Phi'':\X\to\R^{m''}$ tight and slack constraints.
Let $\tilde\Lambda_{\Phi',\Phi''}$ and $\tilde\Ee_{\Hh}^\vee$ be defined as in \Cref{prop:loosely}. If $\hat E$ is a maximal e-variable, then $E\in\tilde\Ee_{\Hh}^\vee$.
\end{lemma}
\begin{proof}
    First, we notice that following exactly the same steps as in the proof of \Cref{lemma:finopt}, we can show that there is $\hat P\in\Hh$ that is fully supported and such that $\langle \hat P,\hat E\rangle =1$. Define $\hat Q\in\PR_\X$ with mass $\hat Q(\{x\}) = \hat E(x)\hat P(\{x\})$. By \Cref{lemma:gru}, we have that $\hat E\in\argmax_{E\in\Ee_\Hh}\langle\hat Q,\log E\rangle$. Moreover, $\supp \hat Q = \supp \hat E = \{x\,:\,E(x)>0\}$. 

    Writing the dual problem for the maximisation of $E\mapsto \langle\hat Q,\log E\rangle$, a solution must be in the form $E_{\lambda}$ (on $\supp\hat Q$), with $\lambda\in\tilde\Lambda_{\Phi',\Phi''}$. (This follows from usual dual Lagrangian arguments, and can for instance be obtained following the reasoning in Appendix C.3 of \citealp{agrawal2020optimal}). By \Cref{lemma:gru}, it follows that $E_\lambda$ and $\hat E$ must coincide on $\supp \hat E$. In particular, $E_\lambda\succeq \hat E$. Since $\hat E$ is maximal, we get $E_\lambda=\hat E$ on the whole $\X$.  
\end{proof}
We can now prove \Cref{prop:loosely}. 
\begin{manual}{\Cref{prop:loosely}}
    Let $\Hh$ be a properly loosely constrained hypothesis on $\X$, with $\Phi':\X\to\R^{m'}$ and $\Phi'':\X\to\R^{m''}$ tight and slack constraints.
    Define $$\tilde\Lambda_{\Phi',\Phi''} = \big\{(\lambda',\lambda'')\;:\;\lambda'\in\R^{m'}\,,\;\lambda''\in[0,+\infty)^{m''}\,,\;1-\lambda'\cdot\Phi'-\lambda''\cdot\Phi''\succeq 0\big\}\,.$$ Then, the optimal e-class for $\Hh$ exists and coincides with $$\tilde\Ee_{\Hh}^\vee = \big\{1-\lambda'\cdot\Phi'-\lambda''\cdot\Phi''\;:\;(\lambda',\lambda'')\in\tilde\Lambda_{\Phi',\Phi''}\big\}\,.$$
\end{manual}
\begin{proof}
    For $\lambda'\in\R^{m'}$ and $\lambda''\in\R^{m''}$, let $E_{\lambda',\lambda''} = 1-\lambda'\cdot\Phi'-\lambda''\cdot\Phi''$. It is clear that $\tilde\Ee_{\Hh}^\vee$ is an e-class. Indeed, for any $(\lambda',\lambda'')\in\tilde\Lambda_{\Phi',\Phi''}$, $E_{\lambda',\lambda''}\succeq 0$ and, for any $P\in\Hh$, we have that $\langle P, E_{\lambda',\lambda''}\rangle = 1 -\lambda'\cdot\langle P,\Phi'\rangle - \lambda''\cdot\langle P, \Phi''\rangle \leq1$, since all the components of $\lambda''$ are non-negative.

    Let $\Phi:\X\to\R^{m'+m''}$ be the continuous function whose first $m'$ components are $\Phi'$ and the last $m''$ are $\Phi''$. Since $\Hh$ is loosely properly constrained, $0\in\relint(\conv\Phi)$. Moreover, note that $\PR_\Phi = \PR_{\Phi'}\cap\PR_{\Phi''}$, so $\Hh' = \{P\in\PR_{\Phi'}\cap\PR_{\Phi''}\,:\,\langle P, \Phi\rangle = 0\}$ is a properly constrained hypothesis, whereof $\Phi$ is a constraint. By construction $\Hh'\subseteq\Hh$, so $\Ee_{\Hh'}\supseteq\Ee_{\Hh}$. Moreover, $\Ee_{\Hh'}^\vee\supseteq\tilde\Ee_{\Hh}^\vee$. In particular, since every $E\in\tilde\Ee_{\Hh}^\vee$ is a maximal e-variable for $\Hh'$, it must also be a maximal e-variable for $\Hh$. Hence, by \Cref{lemma:maxopt} it is enough to show that $\tilde\Ee_{\Hh}^\vee$ is a majorising e-class. 
    
    If $\X$ has finite cardinality, this is \Cref{lemma:loosefin}. The general case is proved by reproducing step by step the proofs of \Cref{prop:compdualopt} and \Cref{thm:dualopt}. Indeed, the results that were used for those proofs are essentially based on the compactness of $\Lambda_{\Phi|S}$ (for suitable choices of $S\subseteq\X$). This property holds also for $\tilde\Lambda_{\Phi',\Phi''|S}$ (with obvious meaning of notation), as $\tilde\Lambda_{\Phi',\Phi''|S}$ is a closed subset of $\Lambda_{\Phi|S}$ (for the same argument as in \Cref{lemma:LFclo}). 
\end{proof}

\end{document}